\documentclass[preprint,11pt]{elsarticle}
\usepackage{amsmath}
\usepackage{amssymb}
\usepackage{amsthm}
\usepackage{graphicx}
\usepackage{epsfig}
\usepackage{color}
\usepackage{subfigure}% subcaptions for subfigures
\usepackage{subfigmat}% matrices of similar subfigures, aka small mulitples
\usepackage{latexsym}
\usepackage{rotating}
\usepackage{amsbsy}
\usepackage{mathrsfs}
\usepackage{lineno}
\usepackage{wrapfig}
\newtheorem{thm}{Theorem}[section]
\newtheorem{prop}{Proposition}[section]
\newtheorem{lem}{Lemma}[section]

\newtheorem{rem}{Remark}[section]

\numberwithin{equation}{section}

\begin{document}

\begin{frontmatter}

\title{$p$ and $hp$ Spectral Element Methods for Elliptic Boundary Layer Problems}
\author[label1]{Akhlaq Husain}
\ead{ahusain10@jmi.ac.in}
\address[label1]{Department of Mathematics, Faculty of Sciences, Jamia Millia Islamia, New Delhi-110025, Delhi, India}
\author[label2]{Aliya~Kazmi\corref{cor}}
\ead{aliya.kazmi.20pd@bmu.edu.in}
\address[label2]{School of Engineering \& Technology, BML Munjal University, Gurgaon-122413, Haryana, India}
\author[label3]{Subhashree~Mohapatra}
\ead{subhashree@iiitd.ac.in}
\address[label3]{Department of Mathematics, Indraprastha Institute of Information Technology Delhi, New Delhi-110020, Delhi, India}
%\author[label4]{Mohammad~Sajid}
%\ead{msajd@qu.edu.sa}
%\address[label4]{Department of Mechanical Engineering, College of Engineering, Qassim University, Buraydah, Saudi Arabia}
\author[label2]{Ziya~Uddin}
\ead{ziya.uddin@bmu.edu.in}
\cortext[cor]{Corresponding author}
\date{}

\begin{abstract}
In this article, we propose $p$ and $hp$ least-squares spectral element methods for one dimensional elliptic
boundary layer problems. Stability estimates are derived and we design numerical schemes based on minimizing
the residuals in the sense of least-squares in appropriate Sobolev norms. We prove parameter robust uniform
error estimates i.e. error in the approximation is independent of the boundary layer parameter $\epsilon$. For
the $p$-version we prove a robust uniform convergence rate of $O\left(\sqrt{\log W}/{W}\right)$, where $W$
denotes the polynomial order used in approximation and for the $hp$-version the convergence rate is shown to
be $O\left(e^{-{W}/{\log W}}\right)$. Numerical results are presented for a number of model elliptic boundary
layer problems confirming the theoretical estimates and uniform convergence results for the $p$ and $hp$ versions.
\end{abstract}


\end{frontmatter}

%----------------------------------------------------------------
\section{Introduction}\label{sec1}
Boundary layers arise as rapidly varying solution components in singularly perturbed elliptic boundary value problems
throughout science and engineering due to the presence of a parameter which may be very small in practice. Examples of
such problems include the convection-diffusion equation, the Navier-Stokes equations in fluid flows with small viscosity,
the Oseen equation, and equations arising in the modeling of semi-conductor devices. They feature diverse spatial scales
(e.g., errors introduced in boundary layers propagate and pollute the solution in the smooth regions) which adds further
difficulties to resolve the boundary layers and complicate traditional numerical approaches. Without proper resolution of
micro-scale structures, standard finite difference and finite element techniques can not provide accurate solutions to
these problems. In some worst cases, these methods may diverge too. Moreover, an important design and implementation aspect
of numerical methods for these problems is robustness, that is, that the performance of the numerical method is independent
or at least fairly insensitive to the parameter. This makes it very difficult to design an efficient and accurate numerical
scheme for this class of problems. A large amount of literature is dedicated to robust numerical methods which have been
proposed and analyzed both theoretically and computationally for a large class of singularly perturbed problems and we refer
to the monographs and expository articles~\cite{DMS,LIN,M2,MRS,MOR,ROO2,RST,SCHW} and their extensive bibliographies for a
great deal of exposition and state of the art research related to singularly perturbed problems.

Most numerical methods employed in the study of singularly perturbed problems are lower order methods based on either finite
difference or lower order finite element methods along with various mesh refinement techniques. In this context we refer
to~\cite{M1,MS3,SCHW1,SSX1} for one-dimensional elliptic boundary layer problems using $p$ and $hp$ finite element methods.
These works contains the proofs of first robust exponential convergence results for a class of boundary layer problems
including convection-diffusion and reaction-diffusion equations. Extension of these results to two-dimensional elliptic
boundary layer problems on smooth domains is available in~\cite{MS1,MS2,SW,XENO1,XENO2} and for problems in non-smooth domains
we refer to~\cite{XENO3}. Finite element methods for singularly perturbed systems of reaction-diffusion equation were considered
in~\cite{XENO4,XO1,XO2,MXO,MX} using the $p$ and $hp$ methods. Extensions to the fourth order elliptic boundary layer problems
is available in~\cite{CVX,CX,CFLX,PZMX} using mixed and $hp$ finite element methods and further consideration on the fourth order
problems containing two small parameters is provided in~\cite{SX,XFS} including the study of eigenvalue problems~\cite{RSX}.
Applications of these methods to singularly perturbed problems of transmission type is studied by Nicaise et al.~\cite{NX1,NX2,NX3,NX4}
and to problems in mechanics (plate and shell problems) by Schwab et al.~\cite{SCHW2}. In~\cite{BS,VHP} robust convergence
estimates were obtained using $h$-version of the finite element method in one dimension. In contrast to the finite element
methods, there have been fewer studies using spectral methods for these problems and we refer to~\cite{CA} for schemes based
on spectral approximations to resolve boundary layers in elliptic problems of Helmholtz and advection-diffusion type. An
efficient spectral Galerkin method for singularly perturbed convection-dominated equation is discussed in~\cite{LS,LT1} which
improves $p$-type estimates of~\cite{CA} using special `mapped' polynomials in one and two dimensions, where singular mappings
of appropriately high order are used to establish algebraic rates of convergence that deteriorate relatively slowly as
$\epsilon\rightarrow 0$. Other interesting and related works for elliptic boundary layer problems using spectral methods are
proposed in~\cite{LT2,TT}.

In the recent years, spectral element methods have been studied and implemented extensively for solving partial differential
equations in several fields because of their flexibility and high order of accuracy (see monographs~\cite{CHQ1,CHQ2,KS,ST} and
references therein). In particular, least-squares spectral element methods are very promising methods which combine the generality
of finite element methods with the accuracy of the spectral methods~\cite{DB,DBG,DKU,DT,DHMU1,DHMU2,DHMU3,KR}. Moreover, these methods
have also the theoretical and computational advantages in the algorithmic design and implementation of the least-squares methods.

The present work is devoted to a detailed analysis of high order spectral element methods, namely, the $p$ and $hp$-version of
the spectral element method (SEM) for one dimensional elliptic boundary layer problems. These problems form a subclass of more
general singularly perturbed problems containing rapidly varying solutions called `boundary layers' (a kind of singularities)
that arise only at the boundary of the domain under consideration. The numerical solution of these problems is of high practical
and mathematical interest. Classical numerical methods which are suitable when the boundary layer parameter (usually denoted by
$\epsilon$) is $O(1)$ become inappropriate as $\epsilon\rightarrow 0$ or is very small, unless the number of degrees of freedom
(say $N$) in one dimension is $N=O(\epsilon^{-1})$, without this, they may fail to resolve boundary layers which are the regions
of foremost interest. The difficulty in applying numerical methods to solve such problems stems from the fact that the derivatives
of these solutions depend on negative powers $\epsilon$ and since estimates for the errors in the approximate numerical solution
obtained by classical schemes depend on bounds for these derivatives, therefore, they are not parameter robust i.e., they do not
hold for small values of $\epsilon$.

In this article we propose a least-squares spectral element method for one dimensional elliptic boundary layer problems using $p$
and $hp$ type discretizations. The method is based on minimizing the sum of the squared $L^2$ norm of the residuals in the
differential equation, the sum of the squared $L^2$ norm of residuals in the boundary conditions and the sum of the jumps in the
function value and its derivatives across the inter-element nodes in $L^2$ norm to enforce the continuity along the inter-element
boundary. The solution is computed by solving normal equations arising from the least-squares formulation using the Preconditioned
Conjugate Gradient Method (PCGM). The resulting linear system is usually ill-conditioned because of the presence of boundary layer
parameter and the condition number of the system deteriorate further for smaller values of the boundary layer parameter and,
therefore, careful design of suitable preconditioners is necessary in order to solve the system using PCGM algorithm in a way that
is robust and efficient with respect to the perturbation parameter. We propose boundary layer preconditioners, similar to the
preconditioners introduced by Dutt et al.~\cite{DBR,DHMU3} for elliptic problems for a spectral element method using a set of
quadratic forms and we also establish a suitable stopping criterion for the solver which is again very critical for saving computational
efforts. The integrals which are involved in the residual computations, are evaluated efficiently by Gauss-quadrature rules. The
proposed method delivers exponential accuracy provided the input data is analytic.

Our main results of this study are the stability estimates, design of numerical scheme and proof of the robust exponential
convergence results for the $p$ and $hp$-versions for a non-conforming spectral element method. It is shown that with a
suitable choice of the spectral element spaces the $p$ version of the method delivers super-exponential convergence when
$W_0>\frac{e}{2\epsilon}$, with $W_0=W+\frac{1}{2}$, where $W$ denotes the polynomial order used for approximation. Indeed,
for the $p$-version, we prove that error between the exact and computed solution is $O\left(\frac{\sqrt{\log W}}{W}\right)$
in a modified parameter dependent $H^2$-norm which is uniform in $\epsilon$. For the $hp$ version, we show that error decays
exponentially and the rate of convergence satisfy an estimate which is $O\left(e^{-{W}/{\log W}}\right)$. We discuss computational
complexity of the method and present results of computational experiments for a number of model elliptic boundary layer problems
for the $p$ and $hp$ versions with homogenous and nonhomogeneous Dirichlet boundary conditions including convection-diffusion
and reaction-diffusion problems for parameter values between $\epsilon=0.1$ to $\epsilon=10^{-4}$. The numerical results validate
the theoretical estimates and comparison of simulation results with those available in the literature are also provided.

This work presents the first robust exponential convergence result for elliptic boundary problems in the framework of
least-squares spectral element methods. The study is motivated by similar works on boundary layer problems by Schwab et
al.~\cite{SCHW1,SSX1} and lies at the intersection of several active research areas having their own distinct approaches
and techniques such as numerical methods for elliptic boundary layer problems, high order numerical methods for elliptic
problems, regularity theory for elliptic boundary layer problems, robust exponential convergence results and least-squares
methods making these ideas a cumulatively a very strong approach for resolving boundary layers.

\section{Elliptic boundary layer problem}\label{sec2}
In one dimension, boundary layers are solution components of the form
\begin{align}\label{2.1}
u_{\epsilon}(x)=\exp{\left(-\frac{x}{\epsilon}\right)},\quad \epsilon>0,\quad  x\in I\subset\mathbb{R},
\end{align}
where $0<\epsilon\leq 1$ is a small parameter called `boundary layer parameter' and $I$ is a bounded interval.
We are interested in designing numerical schemes which will resolve the boundary layers at a robust exponential
rate when (\ref{2.1}) is approximated by polynomials using non-conforming $p$ and $hp$ spectral element methods,
i.e., approximation error is uniform and independent of the parameter $\epsilon$.

Boundary layers given by (\ref{2.1}) arise in many model problems as solution components of singularly
perturbed elliptic boundary value problems. Consider the simple one-dimensional elliptic boundary layer problem
\begin{subequations}
\begin{align}\label{eq2.2}
&\mathcal{L}u=-\epsilon^2 u^{\prime\prime}(x)+u(x)=f(x), \quad x\in \Omega=(a,b)\\
&u(a)=\alpha,\quad u(b)=\beta,
\end{align}
\end{subequations}
where $f\in L^2(\Omega)$. The problem $(2.2a)-(2.2b)$, in general, has boundary layers at both ends of the boundary
i.e. at $x=a,b$ of the form $(\ref{2.1})$.

Let $H^{m}(\Omega)$ denote the usual Sobolev space of $L^{2}(\Omega)$ functions with derivatives of order less than
or equal to $m (m\geq 0)$ in $L^{2}(\Omega)$ equipped with the norm
\[\|u\|^2_{H^{m}(\Omega)}=\|u\|^2_{m,\Omega}=\sum_{|\alpha|\leq m}{\|D^{\alpha}u\|}_{L^{2}(\Omega)}^{2}
=\sum_{|\alpha|\leq m}{\left\|\frac{d^{\alpha}u}{dx^{\alpha}}\right\|}_{L^{2}(\Omega)}^{2},\]
and seminorm
\[|u|^2_{H^{m}(\Omega)}=|u|^2_{m,\Omega}=\sum_{|\alpha|=m}{\|D^{\alpha}u\|}_{L^{2}(\Omega)}^{2}
=\sum_{|\alpha|=m}{\left\|\frac{d^{\alpha}u}{dx^{\alpha}}\right\|}_{L^{2}(\Omega)}^{2},\]
where, $\alpha\in\mathbb{N}_0$, with $\mathbb{N}=\mathbb{N}\cup\{0\}$, $\mathbb{N}$ is the set of positive integers
and $D^{\alpha}u$ is the distributional (weak) derivative of $u$ of order $\alpha$. %Further, let

We define the parameter dependent norm
\begin{align*}
\|u\|^2_{m,\epsilon,\Omega}=\sum_{|\alpha|\leq m}\epsilon^{2m}{\left\|\frac{d^{\alpha}u}{dx^{\alpha}}\right\|}_{L^{2}(\Omega)}^{2},
\end{align*}
and set
\begin{align*}
H^m_{0}(\Omega)=\{u\in H^{m}(\Omega): u=0 \:\mbox{on}\: \partial\Omega\},\quad\quad H^m_{D}(\Omega)=H^{m}(\Omega)\cap \{u|_{\partial\Omega}\}.
\end{align*}
Let $H^m_{\epsilon}(\Omega)$ denote the space
\[H^m_{\epsilon}(\Omega)=\{u\in H^m(\Omega): \|u\|_{m,\epsilon,\Omega}<\infty\}.\]
The following a-priori regularity estimates give existence and uniqueness of the solution of ({2.2a})-({2.2b}) under usual regularity assumptions on the data.
\begin{thm}[Regularity estimate~\cite{M1}]\label{thm1}
If $f\in L^{2}(\Omega)$, then there is a unique solution $u\in H^2_\epsilon(\Omega)$ to the problem (2.2a)-(2.2b) satisfying
\begin{equation}\label{eq2.4}
\|u\|_{2,\Omega,\epsilon}\leq \|f\|_{0,\Omega}+C(\epsilon)\left(|\alpha|+|\beta|\right),\quad C(\epsilon)\leq 2\sqrt{\frac{\epsilon^2}{2}+\frac{2}{3}}.
\end{equation}
\end{thm}
\begin{thm}[Regularity estimate~\cite{SCHW1}]\label{thm2}
If $f\in H^{m}(\Omega)$, then the solution $u\in H^{m+2}(\Omega)\cap H^m_{\epsilon}(\Omega)$ and the shift estimate
\begin{equation}\label{eq2.5}
\|u\|_{m+2,\Omega} \leq C(m,\epsilon)\|f\|_{m,\Omega},\quad m\geq 0,
\end{equation}
holds.
\end{thm}
Note that the above regularity estimates are non-uniform in $\epsilon$ since in the a priori ``shift'' estimates
(\ref{eq2.4})-(\ref{eq2.5}) the constant $C$ depends strongly on the boundary layer parameter $\epsilon$.
\begin{rem}
Schwab et al.~\cite{SCHW1} have shown that, by decomposing the solution $u(x)$ in terms of classical asymptotic expansion
of the solution into a smooth part and boundary layers parts, the regularity of the solution $u(x)$ is determined by the
boundary layer terms when the data $f$ is smooth enough (i.e., $k$ large enough) and the estimate
\begin{equation}\label{eq2.7}
|u_\epsilon(x)|_m\leq \epsilon^{\frac{1}{2}-m}\left(\frac{1-e^{-4/{\epsilon}}}{2}\right)^{1/2}\approx C\left(1+\epsilon^{\frac{1}{2}-m}\right),\quad m=0,1,\cdots,2k,
\end{equation}
holds. Here, the constant $C$ is independent of $\epsilon$ but may depend upon $f$ and $\alpha,\beta$.
\end{rem}

\section{Differentiability and stability estimates}\label{sec3}
In this sections, we present precise estimates for differentiability and stability of solutions and their derivatives
of all orders for elliptic boundary layer problems of the form (2.2a)-(2.2b). Using these differentiability and stability
estimates it will be possible to design a numerical scheme which will deliver robust exponential convergence and efficiency
of computations. Thus, we need to look for properly designed spaces in which precise behavior of the exact solution of
(2.2a)-(2.2b) can be approximated.

\subsection{Differentiability estimates}
Let $u_\epsilon$ be the exact solution of the boundary layer problem (2.2a)-(2.2b). If the input data $f$ is analytic on
$\overline{\Omega}$ then the standard elliptic regularity theory implies that $u_\epsilon$ is analytic in a neighborhood
of $\Omega$ and the following differentiability estimates holds.
\begin{prop}\label{prop2.2.4}
Let $f$ be analytic on $\overline{\Omega}$ then there exists constants $C_\epsilon$ and $d_\epsilon$ depending on $\epsilon$
such that the estimate
\begin{align}\label{3.1}
\underset{{\Omega}}\int\sum_{|\alpha|\leq m}\left|\:D^\alpha u_\epsilon(x)\right|^2 dx \leq C_\epsilon\,(d_\epsilon^m\,m!)^2
\end{align}
holds for all integers $m\geq 1$.
\end{prop}
\begin{proof}
Since $f(x)$ is analytic on $\overline{\Omega}$, therefore, $u_\epsilon$ is analytic in a neighborhood of $\Omega$. Hence (\ref{3.1})
follows by the standard bounds on the derivatives of a real analytic function on a compact set.
\end{proof}
Note that, the constants $C_\epsilon$ and $d_\epsilon$ depend on $\epsilon$ and we need to explicitly control the dependence
of the derivatives of $u_\epsilon$ on the perturbation parameter $\epsilon$ for the design of an efficient numerical scheme.
In this connection, we recall an important differentiability estimate from~\cite{M1}.
\begin{prop}[]
There exists constants $C, d>0$ (independent of $\epsilon$) depending only on $f$, $\alpha$ and $\beta$ such that the estimate
\begin{align}\label{3.2}
\underset{{\Omega}}\int\sum_{|\alpha|\leq m}\left|\:D^\alpha u_\epsilon(x)\right|^2 dx
\leq Cd^{2m}\left(\max{\left(m,\epsilon^{-1}\right)^m}\right)^2
\end{align}
holds for all integers $m\geq 1$.
\end{prop}
The above Proposition says that the $m$th derivative of the solution $u_\epsilon$ will be independent of $\epsilon$
provided $m\geq \frac{1}{\epsilon}$. Roughly speaking, this amounts to say that the boundary layers arising for small
values of the perturbation parameter $\epsilon$ are not visible for derivatives of sufficiently large order. From this
we can obtain robust exponential convergence for the $p$-version of the spectral element method provided that the
polynomial order $W\approx O\left(\frac{1}{\epsilon}\right)$. Note that, the differentiability estimate (\ref{3.2})
does not capture the boundary layer behavior of the solution accurately for derivatives of order $m<\frac{1}{\epsilon}$.
More precisely, the exponential convergence is only visible if $W>\frac{1}{\epsilon}$ and in the practical range
$W\ll \frac{1}{\epsilon}$, we will only observe convergence of the type $O\left(\frac{\sqrt{\log W}}{W}\right)$ (see
Theorem 5.1) which is same as derived in~\cite{M1,M2,MS3,SCHW,SCHW1,SSX1}. This degradation in the performance of the
spectral method in the practical range of interest is due to the presence of boundary layers.

\subsection{Discretization and spectral element representation}
We divide the domain $\Omega=(a,b)$ into $N$ elements $\Omega_l=(x_{l-1},x_l)$, $l=1,2,\ldots,N$ of equal size for simplicity
(see Figure \ref{fig1}) such that $a=x_0<x_1<\ldots<x_N=b$. Let $\frac{(b-a)}{N}=h$ denotes the mesh size. We remark that, our
analysis and estimates remain true when the nodes $x_l$, $l=1,\cdots,N$ are not equally spaced.
\begin{figure}[ht]
\centering
\scalebox{0.60}[0.60]{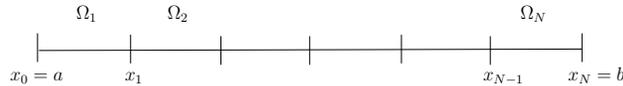}
\caption{Discretization of the domain $\Omega$ (uniform refinement)}
\label{fig1}
\end{figure}

A set of non-conforming spectral element functions are now defined on the elements $\Omega_l$ via polynomials
of degree $W$, where $W$ denote an upper bound on the degree of the polynomial representation of the spectral
element functions. Let $M_l$ denote the analytic (bijective) map from the reference interval $I_0=(-1,1)$ to
$\Omega_l$ defined by
\begin{align}\label{3.3}
x &=\left(\frac{1-\xi}{2}\right)x_{l-1}+\left(\frac{1+\xi}{2}\right)x_l,
\end{align}
having an analytic inverse. Note that $M_l^{-1}$ maps the $l$-th element $\Omega_l=(x_{l-1},x_l)$ to the
interval $I_0$. On each of the mapped interval $I_0$, we define our nonconforming spectral element functions
as polynomials of degree $W$ denoted by $\hat{u}_l$ and defined by
\begin{align}\label{3.4}
\hat{u}_l(\xi) &=\sum_{i=0}^{W}\alpha_i^{l}\xi^{i}.
\end{align}
Here, $\alpha_i^{l}$ are unknowns to be determined. Then the spectral element functions $u_l(x)$ on $\Omega_l$
are defined by
\begin{align}\label{3.5}
u_l(x) &=\hat{u}_l\left(M_l^{-1}(x)\right)
\end{align}
for $l=1,2,\ldots,N$.

\subsection{Stability estimate}
Let $\mathcal{S}^{W}(\Omega)=\{u_l(x)\}_l$ be the space of spectral element functions on $\Omega$ which are defined elementwise
by (\ref{3.5}) as polynomials of degree less than or equal to $W$.
To state the stability theorem we need to define some quadratic forms. Since the spectral element functions in our analysis
are nonconforming, therefore we need to consider the jumps in the function $u$ and its derivative at all the inter element
nodes to enforce the continuity at the inter element nodes.

Let $x_l$ denotes the inter element node common to two neighbours $\Omega_l$ and $\Omega_{l+1}$. Let $[u]|_{x_l}$ denote
the jump in $u$ across the inter element node $x_l$. Then we may assume that $x_l$ is the image of the right end point
$\xi=1$ of the reference interval $I_0$ under $M_l$ which takes $I_0$ to $\Omega_l$ and at the same time  $x_l$ is also
the image of the left end point $\xi=-1$ of the reference interval $I_0$ under the map $M_{l+1}$ which takes $I_0$ to
$\Omega_{l+1}$. Now using (\ref{3.3}), we can write
\begin{align*}
\frac{du_l}{dx}(x_l) &=\frac{d\hat{u}_l}{d\xi}(1)\frac{d\xi}{dx}=\frac{2}{h}\frac{d\hat{u}_l}{d\xi}(1), \\
\frac{du_{l+1}}{dx}(x_l) &=\frac{d\hat{u}_{l+1}}{d\xi}(-1)\frac{d\xi}{dx}=\frac{2}{h}\frac{d\hat{u}_{l+1}}{d\xi}(-1).
\end{align*}
Therefore, the jump terms simplify to
\begin{subequations}\label{3.6}
\begin{align}
\left|[u(x_l)]\right| &=|u_l(x_l)-u_{l+1}(x_l)|=|\hat{u}_{l}(1)-\hat{u}_{l+1}(-1)|, \\
\left|\left[\frac{du}{dx}(x_l)\right]\right| &=\left|\frac{du_l}{dx}(x_l)-\frac{du_{l+1}}{dx}(x_l)\right|
=\frac{2}{h}\left|\frac{d\hat{u}_l}{d\xi}(1)-\frac{d\hat{u}_{l+1}}{d\xi}(-1)\right|.
\end{align}
\end{subequations}
Finally, we consider the boundary conditions at the points $x=a$ and $x=b$. Since $x=a$ is the image of $\xi=-1$
under the mapping $M_1$ and $x=b$ is the image of $\xi=1$ under the mapping $M_N$, therefore
\begin{align}\label{3.7}
|u(a)| &=|\hat{u}_{1}(-1)| \quad \mbox{and}\quad |u(b)|=|\hat{u}_{N}(1)|.
\end{align}
Consider the $p-$version of the method with only one element in the decomposition of $\Omega$, i.e., $N=1$ and
define the quadratic form
\begin{align}\label{3.8}
\mathcal{V}_p^W(\{u(x)\}) &=\|\mathcal{L}u(x)\|_{0,\Omega}^{2}+|u(a)|^{2}+|u(b)|^{2}.
\end{align}
Next, consider the $hp-$version of the method with a variable mesh, i.e., $N>1$ and assume that $N$ is proportional
to $W$. Now define the quadratic form
\begin{align}\label{3.9}
\mathcal{V}_{hp}^W(\{u_l(x)\}_l) &=\sum_{l=1}^{N}\|\mathcal{L}u_l(x)\|_{0,\Omega_l}^{2}
+\sum_{l=1}^{N-1}\left(\left|[u(x_l)]\right|^{2}+\left|\left[\frac{du}{dx}(x_l)\right]\right|^{2}\right)+|u(a)|^{2}+|u(b)|^{2}.
\end{align}
For proving the stability estimate we shall make use of the following Lemma, the proof of which is similar to the proof
of Lemma 2.1 of~\cite{KJ} and hence omitted.
%-------------------------------------------
\begin{lem}\label{lem1}
Let $\{u_l(x)\}_l\in\mathcal{S}^{W}(\Omega)$. Then there exists a function $(\{v_l(x)\}_l)$ such that $v_{l}(a)=0$,
$v_{l}(b)=0$, $v_{l}\in H^{2}(\Omega_{l})$ and $u+v(=w)\in H^{2}(\Omega)$. Moreover, there exists a constant $C_\epsilon$ (depending
on $\epsilon$) such that
\begin{align}\label{3.10}
\sum_{l=1}^{N}||v_{l}||_{2,\epsilon,\Omega_{l}}^{2} & \le C_\epsilon\sum_{l=1}^{N-1}\left(|[u(x_{l})]|^{2}
+\left|\left[\frac{du}{dx}(x_{l})\right]\right|^{2}\right),\quad\mbox{provided}\quad W=O\left(\frac{1}{\epsilon}\right).
\end{align}
\end{lem}
%=================================================================
We are now in a position to prove the stability estimates for the $p$ and $hp$-versions of the least-squares spectral
element method examined in this paper. First we state the stability theorem for the $p$-version of the method.
\begin{thm}\label{thm3.1}
Consider the elliptic boundary layer problem (2.2a)-(2.2b). If $p$-version of the spectral element method is used then
there exists a constant $C_\epsilon>0$ depending on $\epsilon$ such that
\begin{align}\label{3.11}
\|u\|^{2}_{2,\epsilon,\Omega} &\leq C_\epsilon\mathcal{V}_p^{W}(\{u(x)\}),
\end{align}
provided $W=O\left(\frac{1}{\epsilon}\right)$.
\end{thm}
\begin{proof}
By Lemma $\ref{lem1}$, there exists $v(x)$ such that $v(a)=v(b)=0$ and $u+v=w\in H^{2}(\Omega)$. So that for
$W=O\left(\frac{1}{\epsilon}\right)$, we have
\begin{align}\label{3.12}
\|w\|^{2}_{2,\epsilon,\Omega} &\leq C_\epsilon(\|\mathcal{L}w\|^{2}_{0,\Omega}+|w(a)|^{2}+|w(b)|^{2}).
\end{align}
Now,
\begin{align*}
\|\mathcal{L}w\|^{2}_{0,\Omega} &=\|\mathcal{L}(u+v)\|^{2}_{0,\Omega} \leq 2\|\mathcal{L}u\|^{2}_{0,\Omega}
+2\|\mathcal{L}v\|^{2}_{0,\Omega}\leq C_\epsilon\left(\|\mathcal{L}u\|^{2}_{0,\Omega}+\|v\|^{2}_{2,\epsilon,\Omega}\right).
\end{align*}
From this, we can write
\begin{align}\label{3.13}
\|w\|^{2}_{2,\epsilon,\Omega} &\leq C_\epsilon\left(\|\mathcal{L}u\|^{2}_{0,\Omega}+\|v\|^{2}_{2,\epsilon,\Omega}
+|u(a)|^{2}+|u(b)|^{2}\right).
\end{align}
since $w=u+v$ and $v(a)=v(b)=0$.
\newline
There exists a constant $K$ such that
\begin{align}\label{3.14}
\|u\|^{2}_{2,\epsilon,\Omega} &\leq K\left(\|w\|^{2}_{2,\epsilon,\Omega}+\|v\|^{2}_{2,\epsilon,\Omega}\right).
\end{align}
Using Lemma \ref{lem1} for $N=1$ and (\ref{3.13}) in (\ref{3.14}), we get
\begin{align*}
\|u\|^{2}_{2,\epsilon,\Omega} &\leq C_\epsilon\left(\|\mathcal{L}u\|^{2}_{0,\Omega}+|u(a)|^{2}+|u(b)|^{2}\right)
\leq C_\epsilon\mathcal{V}_p^{W}(\{u(x)\}).
\end{align*}
\end{proof}
Next, we state the corresponding result for $hp$-version of the method.
\begin{thm}\label{thm3.2}
Consider the elliptic boundary layer problem (2.2a)-(2.2b). If $hp$-version of the spectral element method is used
then there exists a constant $C_\epsilon>0$ depending on $\epsilon$ such that the estimate
\begin{align}\label{3.15}
\sum_{l=1}^{N}\|u_{l}\|^{2}_{2,\epsilon,\Omega_{l}} &\leq C_\epsilon\mathcal{V}_{hp}^{W}(\{u_{l}(x)\}_{l})
\end{align}
holds, provided $W=O\left(\frac{1}{\epsilon}\right)$.
\end{thm}
%================================================================
\begin{proof}
The proof is similar to the proof of the stability Theorem 2.2 in~\cite{KJ}. By Lemma $\ref{lem1}$, there exists
$\{v_{l}(x)\}$ such that $v_{l}(a)=v_{l}(b)=0$ and $u+v=w\in H^{2}(\Omega)$. So for $w\in H^{2}(\Omega)$,
\begin{align}\label{3.16}
\|w\|^{2}_{2,\epsilon,\Omega} &\le C_\epsilon(\|\mathcal{L}w\|^{2}_{0,\Omega}+|w(a)|^{2}+|w(b)|^{2}),
\end{align}
provided $W=O\left(\frac{1}{\epsilon}\right)$. Now,
\begin{align*}
\|\mathcal{L}w\|^{2}_{0,\Omega} &=\|\mathcal{L}(u+v)\|^{2}_{0,\Omega} \leq 2\|\mathcal{L}u\|^{2}_{0,\Omega}+2\|\mathcal{L}v\|^{2}_{0,\Omega}
\leq C_\epsilon\left(\sum_{l=1}^{N}\|\mathcal{L}u_{l}\|^{2}_{0,\Omega_{l}}+\sum_{l=1}^{N}\|v_{l}\|^{2}_{2,\epsilon,\Omega_{l}}\right).
\end{align*}
Using Lemma $\ref{lem1}$ in the last term of the above inequality, we get
\begin{align*}
\|\mathcal{L}w\|^{2}_{0,\Omega} &\leq C_\epsilon\left(\sum_{l=1}^{N}\|\mathcal{L}u_{l}\|^{2}_{0,\Omega_{l}}
+\sum_{l=1}^{N-1}\left(|[u(x_l)]|^{2}+\left|\left[\frac{du}{dx}(x_{l})\right]\right|^{2}\right)\right).
\end{align*}
From this, we get
\begin{align}\label{3.17}
\|w\|^{2}_{2,\epsilon,\Omega} &\leq C_\epsilon\left(\sum_{l=1}^{N}||\mathcal{L}u_{l}||^{2}_{0,\Omega_{l}}
+\sum_{l=1}^{N-1}\left(|[u(x_l)]|^{2}+\left|\left[\frac{du}{dx}(x_{l})\right]\right|^{2}\right)+|u(a)|^{2}+|u(b)|^{2}\right).
\end{align}
since $w=u+v$ and $v(a)=v(b)=0$.
\newline
There exists a constant $K$ such that
\begin{align}\label{3.18}
\sum_{l=1}^{N}\|u\|^{2}_{2,\epsilon,\Omega_{l}} &\leq K\left(\|w\|^{2}_{2,\epsilon,\Omega}+\sum_{l=1}^{N}\|v\|^{2}_{2,\epsilon,\Omega_{l}}\right).
\end{align}
Therefore, using Lemma \ref{lem1} and (\ref{3.17}) in (\ref{3.18}) we obtain
\begin{align*}
\sum_{l=1}^{N}\|u\|^{2}_{2,\epsilon,\Omega_{l}} &\leq C_\epsilon\left(\sum_{l=1}^{N}\|\mathcal{L}u_{l}\|^{2}_{0,\Omega_{l}}
+\sum_{l=1}^{N-1}\left(|[u(x_l)]|^{2}+\left|\left[\frac{du}{dx}(x_{l})\right]\right|^{2}\right)+|u(a)|^{2}+|u(b)|^{2}\right)\\
&\leq C_\epsilon\mathcal{V}_{hp}^{W}(\{u_{l}(x)\}_{l}).
\end{align*}
\end{proof}

\section{Numerical Scheme and Preconditioners}\label{sec4}
In this section, we describe numerical schemes based on the stability estimates (Theorems $\ref{thm3.1}$ and $\ref{thm3.2}$)
proved in Section \ref{sec3} and discuss preconditioning techniques for the $p$ and $hp$-versions.

\subsection{Numerical Scheme}
In Section \ref{sec3}, we had divided the domain $\Omega=(a,b)$ into $N$ elements denoted by $\Omega_l$, $l=1,2,\cdots,N$
and introduced a spectral element representation of the function $u$ on each of the elements $\Omega_l$ in the subdivision.

To formulate the numerical scheme for the $p$ and $hp$-versions we shall define two functionals $\mathcal{R}_p^{W}(\{u(x)\})$
and $\mathcal{R}_{hp}^{W}(\{u_{l}(x)\}_{l})$, closely related to the quadratic forms $\mathcal V_p^{W}(\{u(x)\})$ and
$\mathcal V_{hp}^{W}(\{u_{l}(x)\}_{l})$ respectively. Let $f_{l}(x)=f|_{\Omega_{l}}$, $l=1,2,\ldots,N$.

First, we define a functional for the $p$-version by
\begin{align}\label{4.1}
\mathcal{R}_p^{W}(u(x)) &=\|\mathcal{L}u(x)-f\|^{2}_{0,\Omega}+|u(a)-\alpha|^{2}+|u(b)-\beta|^{2}.
\end{align}
Our numerical scheme for the $p$-version may now be formulated as follows:
\begin{prop}
Find $z\in \mathcal{S}^{W}(\Omega)$ which minimizes the functional $\mathcal{R}_p^{W}(u)$ over all $u\in\mathcal{S}^{W}(\Omega)$.
\end{prop}
Thus, our approximate solution is the unique $z\in \mathcal{S}^{W}(\Omega)$, which minimize the functional $\mathcal{R}_p^{W}(u)$
over all $u\in \mathcal{S}^{W}(\Omega)$. In other words,

\emph{The numerical scheme seeks a solution which minimizes the sum of the squared norm of the residuals in the differential
equation and the sum of the squares of the residuals in the boundary conditions}.

Next, we define a functional for the $hp$-version by
\begin{align}\label{4.2}
\mathcal{R}_{hp}^{W}(\{u_{l}\}_{l}) &=\sum_{l=1}^{N}\|\mathcal{L}u_{l}(x)-f_{l}\|^{2}_{0,\Omega_{l}}
+\sum_{l=1}^{N-1}\left(|[u(x_{l})]|^{2}+\left|\left[\frac{du}{dx}(x_{l})\right]\right|^{2}\right)\notag\\
&+|u(a)-\alpha|^{2}+|u(b)-\beta|^{2}.
\end{align}
Thus, the numerical scheme for the $hp$-version may now be formulated as follows:
\begin{prop}
Find $\{z_l\}_l\in \mathcal{S}^{W}(\Omega)$ which minimizes the functional $\mathcal{R}_{hp}^{W}(\{u_l\}_l)$ over all $\{u_l\}_l\in\mathcal{S}^{W}(\Omega)$.
\end{prop}

Thus, the desired spectral element solution is the unique $(\{z_l\}_l)\in \mathcal{S}^{W}(\Omega)$, which minimize $\mathcal{R}_p^{W}(\{u_l\}_l)$
over all $\{u_l\}_l\in \mathcal{S}^{W}(\Omega)$. In other words,

\emph{The numerical scheme seeks a solution which minimizes the sum of the squared norm of the residuals in the differential
equation and the sum of the squares of the residuals in the boundary conditions and enforce continuity by adding a term which
measures the sum of the squared norms of the jump in the function and its derivatives at the inter element nodes}.

Note that the numerical schemes for the $p$ and $hp$-versions are essentially defined as a least-squares formulation.
The normal equations arising from this formulation will be solved using the preconditioned conjugate gradient method
(PCGM).
The calculation of integrals in the residual computations is done via Gauss-Lobatto-Legendre (GLL) quadratures and
similar to one dimensional problems as described in~\cite{DB,KJ}, therefore we omit the details here.

\subsection{Preconditioners}
The system of linear equations obtained from the least-squares formulation is symmetric and positive definite but depend
on the boundary layer parameter, which may be arbitrarily small, and therefore the resulting linear system is usually
ill-conditioned. This will slow down the convergence of the PCGM which is used as a solver to solve this system. Therefore
a suitable preconditioner is required to obtain the solution which is efficient and robust, with respect to the perturbation
parameter. We now describe boundary layer preconditioners for the $p$ and $hp$-versions which are similar to those presented
in~\cite{DBR}. More specifically, we seek a matrix $M$ that is spectrally equivalent (see~\cite{SCHW}, Section 4.7) to the
matrix $A^TA$ i.e.,
\[C_1 \langle A^TA u, u \rangle \leq \langle M u, u \rangle \leq C_2 \langle A^TA u, u \rangle.\]
Thus, we can use $M$ as a preconditioner and the condition number of the preconditioned system satisfies
\[\kappa(M^{-1}A^TA) \leq \frac{C_2}{C_1}.\]
We define a norm equivalent quadratic form $\mathcal{U}_p^{W}(u(x))$ by
\begin{align}\label{4.3}
\mathcal{U}_p^{W}(u(x)) &=\epsilon^4\|u^{\prime\prime}\|^{2}_{0,\Omega}+\|u\|^{2}_{0,\Omega}.
\end{align}
Using the stability estimate Theorem $3.1$, it follows that
\begin{align}\label{4.4}
\mathcal{U}_p^{W}(u(x)) &\leq C_\epsilon\mathcal{V}_p^{W}(u(x)),
\end{align}
where $C_\epsilon=\frac{1}{1-\epsilon^2}$. Now, using Cauchy-Schwarz and Young's inequality
\begin{align*}%\label{4.5}
\|\mathcal{L}u\|^{2}_{0,\Omega} &\leq \epsilon^4\|u^{\prime\prime}\|^{2}_{0,\Omega}+2\epsilon^2\|u^{\prime\prime}\|_{0,\Omega}\|u\|_{0,\Omega}+\|u\|^{2}_{0,\Omega}\\
&\leq \epsilon^4\|u^{\prime\prime}\|^{2}_{0,\Omega}+\epsilon^2\left(\epsilon^2\|u^{\prime\prime}\|^2_{0,\Omega}+\frac{1}{\epsilon^2}\|u\|^2_{0,\Omega}\right)+\|u\|^{2}_{0,\Omega},\\
%&\leq (\epsilon^4+\epsilon^2)\|u^{\prime\prime}\|^{2}_{0,\Omega}+(1+\epsilon^2)\|u\|^{2}_{0,\Omega},\\
&\leq \left((1+\epsilon^2)\epsilon^4\|u^{\prime\prime}\|^{2}_{0,\Omega}+\left(1+\frac{1}{\epsilon^2}\right)\|u\|^{2}_{0,\Omega}\right)\\
&\leq \max\left(1+\epsilon^2,1+\frac{1}{\epsilon^2}\right)\left(\epsilon^4\|u^{\prime\prime}\|^{2}_{0,\Omega}+\|u\|^{2}_{0,\Omega}\right).
\end{align*}
Thus, there exists a constant $K_1=1+\frac{1}{\epsilon^2}$ (independent of $W$) such that
\begin{align}\label{4.5}
\|\mathcal{L}u\|^{2}_{0,\Omega}\leq K_1\mathcal{U}_p^{W}(u(x)).
\end{align}
Moreover,
\begin{align}\label{4.6}
|u(a)|\leq \|u\|_{0,\Omega}, \quad |u(b)|\leq \|u\|_{0,\Omega}.
\end{align}
Combining (\ref{4.5}) and (\ref{4.6}), it follows that there exists a constant $K_2$ such that
\begin{align}\label{4.7}
\mathcal{V}_p^{W}(u(x)) &\leq K_{2}\mathcal{U}_p^{W}(u(x))
\end{align}
where $K_{2}=K_{1}+2$. Combining $(\ref{4.4})$ and $(\ref{4.7})$ we obtain,
\begin{align}\label{4.8}
K_\epsilon\mathcal{V}_p^{W}(u(x)) &\le \mathcal{U}_p^{W}(u(x))\leq C_\epsilon \mathcal{V}_p^{W}(u(x))
\end{align}
for all $u\in \mathcal{S}^{W}(\Omega)$, where $K_\epsilon=\frac{1}{K_2}$ is a constant. Thus, the quadratic forms $\mathcal{V}_p^{W}(u(x))$
and $\mathcal{U}_p^{W}(u(x))$ are spectrally equivalent. So we can use the quadratic form $\mathcal{U}_p^{W}(u(x))$ as a preconditioner for
the $p$-version. Moreover, the condition number of the preconditioned system will be $\kappa(M^{-1}A^TA)=\frac{C_\epsilon}{K_\epsilon}=O\left(\frac{1}{\epsilon^2}\right)$
provided $\epsilon\ll 1$.

Next, we consider preconditioners for the $hp$-version. We define the quadratic form,
\begin{align}\label{4.9}
\mathcal{U}_{hp}^{W}(\{u_{l}(x)\}_{l}) &=\sum_{l=1}^{N}\left(\epsilon^4\|u_l^{\prime\prime}\|^{2}_{0,\Omega_l}+\|u_l\|^{2}_{0,\Omega_l}\right).
\end{align}
Using Theorem $3.2$, it follows that
\begin{align}\label{4.10}
\mathcal{U}_{hp}^{W}(\{u_{l}(x)\}_{l}) &\leq C_\epsilon\mathcal{V}_{hp}^{W}(\{u_{l}(x)\}_{l}).
\end{align}
where $C_\epsilon=\frac{1}{1-\epsilon^2}$. At the same time using Cauchy-Schwarz and Young's inequality we can find positive constants $L_{1}$
and $L_2$ such that
\begin{align*}%\label{4.5}
\sum_{l=1}^{N}\|\mathcal{L}u_l\|^{2}_{0,\Omega_l} &\leq \sum_{l=1}^{N}\left(\epsilon^4\|u_l^{\prime\prime}\|^{2}_{0,\Omega_l}+2\epsilon^2\|u_l^{\prime\prime}\|_{0,\Omega_l}\|u_l\|_{0,\Omega_l}+\|u_l\|^{2}_{0,\Omega_l}\right)\\
&\leq \sum_{l=1}^{N}\left((1+\epsilon^2)\epsilon^4\|u_l^{\prime\prime}\|^{2}_{0,\Omega_l}+\left(1+\frac{1}{\epsilon^2}\right)\|u_l\|^{2}_{0,\Omega_l}\right).
\end{align*}
From (\ref{4.9}) and the above estimate, it follows that there exists a constant $L_1=1+\frac{1}{\epsilon^2}$ (independent of $W$) such that
\begin{align}\label{4.11}
\sum_{l=1}^{N}\|\mathcal{L}u_{l}\|^{2}_{0,\Omega_{l}} \leq L_{1}\mathcal{U}_{hp}^{W}(\{u_{l}(x)\}_{l}).
\end{align}
Using continuous Sobolev embedding results for bounded domains in one dimension~\cite{ADAM} (Chapter 4), there exist positive constants
$L_{2}$ and $L_{3}$ such that
\begin{subequations}\label{4.12}
\begin{align}
\sum_{l=1}^{N-1}\left|[u(x_l)]\right|^{2}&\leq L_{2}\|u_l\|^{2}_{1,\epsilon,\Omega_l},\quad \sum_{l=1}^{N-1}\left|\left[\frac{du}{dx}(x_l)\right]\right|^{2}
\leq L_{3}\|u_l\|^{2}_{2,\epsilon,\Omega_l},\label{4.12a}\\
|u(a)|^2&\leq \|u_l\|^{2}_{0,\Omega_l}, \qquad\qquad\quad\: |u(b)|^2\leq \|u_l\|^{2}_{0,\Omega_l}\label{4.12b}.
\end{align}
\end{subequations}
Combining the estimates (\ref{4.11}), (\ref{4.12a}) and (\ref{4.12b}) we obtain
\begin{align}\label{4.13}
\mathcal{V}_{hp}^{W}(\{u_{l}(x)\}_{l}) &\leq L_{4}\mathcal{U}_{hp}^{W}(\{u_{l}(x)\}_{l}),
\end{align}
where $L_{4}=\sum_{i=1}^{3}L_{i}+2$. Combining $(\ref{4.10})$ and $(\ref{4.13})$ it follows that there exists a constant $L_\epsilon$ such that
\begin{align}\label{4.14}
L_\epsilon\mathcal{V}_{hp}^{W}(\{u_{l}(x)\}_{l}) &\leq \mathcal{U}_{hp}^{W}(\{u_{l}(x)\}_{l})\le C_\epsilon \mathcal{V}_{hp}^{W}(\{u_{l}(x)\}_{l})
\end{align}
for all $\{u_{l}(x)\}_{l}\in \mathcal{S}^{W}(\Omega)$, where $L_\epsilon=\frac{1}{L_4}$ is a constant that depends on $\epsilon$ but independent
of $W$. Thus, the two forms $\mathcal{V}_{hp}^{W}(\{u_{l}(x)\}_{l})$ and $\mathcal{U}_{hp}^{W}(\{u_{l}(x)\}_{l})$ are spectrally equivalent.
So we use the quadratic form $\mathcal{U}_{hp}^{W}(\{u_{l}(x)\}_{l})$ as a preconditioner for the $hp$-version. Now, the condition number of
the preconditioned system is $\kappa(M^{-1}A^TA)=\frac{C_\epsilon}{L_\epsilon}=O\left(\frac{1}{\epsilon^2}\right)$ provided $\epsilon\ll 1$.

It follows from the spectral equivalences in (\ref{4.8}) and (\ref{4.10}) that the condition number of the preconditioned system
behaves like $O\left(\frac{1}{\epsilon^2}\right)$ for the $p$ and $hp$ versions, therefore the preconditioners will be very effective
in resolving boundary layers. Notice that the condition number is independent of $W$ for both the $p$ and $hp$ versions due to the
spectral basis and the fact that the quadratic forms $\mathcal{U}_{p}^{W}(\{u(x)\})$ or $\mathcal{U}_{hp}^{W}(\{u_{l}(x)\}_{l})$
are local to each element but the jump terms ensure continuity at the inter element nodes, therefore, the norm equivalence constants
$C_1$, $C_2$ remain independent of $N$ and $W$.

Our purpose is to solve the normal equations $A^TAU=A^Tb$ using the PCGM therefore, we need to obtain a suitable stopping
criteria for the PCGM algorithm. The ideas presented here are motivated from~\cite{MM,MMN,NMM} which were concerned with a
similar stopping criterion for boundary layer preconditioners for singularly perturbed convention-diffusion as well as
reaction-diffusion problems in the framework of finite difference and finite element methods using maximum norm and energy
norm estimates. We now adapt the same argument for spectral element discretizations and $\epsilon-$ norm estimates.

Let $u_{\mbox{exact}}$ be the exact solution and $u_{\mbox{sem}}$ be approximate (spectral element) solution obtained from
the SEM discretizations. Typically, the PCGM iterations are terminated as soon as we reach the `best' approximation i.e., an
iterate $u^{(k)}$, which approximates $u_{\mbox{sem}}$ to a desired degree of accuracy that is usually of the same order as
the discretization error $\|u_{\mbox{exact}}-u_{\mbox{sem}}\|_{2,\epsilon,\Omega}$. Thus, we have
\begin{align}\label{4.15}
\|u_{\mbox{exact}}-u^{(k)}\|_{2,\epsilon,\Omega}&\leq \|u_{\mbox{exact}}-u_{\mbox{sem}}\|_{2,\epsilon,\Omega}
+\|u_{\mbox{sem}}-u^{(k)}\|_{2,\epsilon,\Omega} \notag\\
&\leq C\|u_{\mbox{exact}}-u_{\mbox{sem}}\|_{2,\epsilon,\Omega}
\end{align}
for some constant $C$.

Let $E^{(k)}=u_{\mbox{sem}}-u^{(k)}$ denote the error in approximation at the $k$-th iteration. In PCGM, the best way to
measure this error is the residual
\[R^{(k)}=F-Au^{(k)}=F-A\left(u_{\mbox{sem}}-E^{(k)}\right)=AE^{(k)}\]
obtained at the $k$-th step. Therefore, the discretization error is $E^{(k)}=A^{-1}R^{(k)}$, where $A$ is symmetric and
positive definite matrix. Therefore, $A^{-1/2}$ is also positive definite so that, $\|A^{-1/2}\|_{2}=\sqrt{\|A^{-1}\|_{2}}$
because the eigenvalues of $A^{-1}$ are the squares of those of $A^{-1/2}$. Also, the simplest bound on $E^{(k)}$ to achieve
discretization accuracy is
\[\|E^{(k)}\|_{\infty}=\|A^{-1}\|_{\infty}\:\|R^{(k)}\|_{\infty}.\]
It is known~\cite{MM} that on Shishkin meshes the bound on the residual is
\[\|R^{(k)}\|_{\infty}\leq C\frac{\epsilon^2 (\log W)^3}{W^3}\]
while that on Bakhvalov meshes is
\[\|R^{(k)}\|_{\infty}\leq C\frac{\epsilon^2}{W^3}.\]
Note that when $\epsilon$ is small and polynomial order $W$ is moderate then $||R^{(k)}||_{\infty}$ will be much less than
machine precision and therefore, as a result $||R^{(k)}||_{\infty}$ can not be computed accurately which will lead to
inaccurate calculations in the discretization error.

To overcome this difficulty, we apply the preconditioned conjugate gradient algorithm and make use of the standard stopping
criterion that bounds the inner product of $R^{(k)}$ with the preconditioned residual, $Z^{(k)}=MR^{(k)}$, where $M$ is the
preconditioning matrix. If $M^{-1}$ is a `good' preconditioner for $A$ in the sense that is spectrally equivalent to the identity
matrix $I$, i.e., $M^{-1}A\approx I$ then
\[\left(Z^{(k)}\right)^T\: R^{(k)}=\left(E^{(k)}\right)^T\: R^{(k)}AMAE^{(k)}\approx ||E^{(k)}||_{A}.\]
Now,
\begin{align}\label{4.16}
\|E^{(k)}\|_{A}&=\sqrt{\left(E^{(k)}\right)^TAE^{(k)}}=\sqrt{\left(E^{(k)}\right)^TA^TA^{-1}AE^{(k)}}=\|A^{-1}R^{(k)}\|_2\leq \|A^{-1}\|_2\: \|R^{(k)}\|_2.
\end{align}
Note that, the natural estimation and minimization of $\|E^{(k)}\|_{A}$ by the PCG algorithm makes it a natural stopping
criterion for the solver. Moreover, we have~\cite{MM}
\begin{align}\label{4.17}
\|A^{-1/2}\|_2 \leq kN\approx kW,
\end{align}
where, $k>0$ is a constant. Hence, to guarantee that
$\|E^{(k)}\|_{A}\leq C\|u_{\mbox{exact}}-u_{\mbox{sem}}\|_{2,\epsilon,\Omega}$, we have
\begin{align}\label{4.18}
\|R^{(k)}||_2\leq C\frac{(\log W)^{1/2}}{W}.
\end{align}
Finally, we remark that although the bound on $\|E^{(k)}\|_{A}$ is dependent on $W$, the number of iterations for convergence
is bounded, at worst, by a logarithmic factor of $W$, which confirms that in practice iteration count grows slowly.

\section{Error Estimates}\label{sec5}
Now we prove robust error estimates for the non-conforming SEM with analytic input data for the $p-$version with a fixed
mesh and for the $hp-$version of the method with a variable mesh. In order to prove the main theorem on error estimates
we need to consider some approximation results. Let $I_0=(-1,1)$ and $\hat{u}\in L^2(I_0)$ and let
\begin{align}\label{5.1}
\hat{u}(\xi)=\sum_{n=0}^{\infty} a_nL_n(\xi),
\end{align}
be the Legendre series expansion of $\hat{u}$, where
\begin{align}\label{5.2}
a_n=\frac{2n+1}{2}\int_{-1}^{1}\hat{u}(\xi)L_n(\xi)d\xi,\quad \|\hat{u}\|^2_{L^2(I_0)}=\sum_{n=0}^{\infty} \frac{2}{2n+1}|a_n|^2.
\end{align}
In order to prove the main error estimate theorem we need some definitions and approximation results which are stated now.

For any $0\leq j\leq k$, we recall the space $S^{k}_j(I_0)$~\cite{SCHW}, which is defined to be the space of functions
$u\in L^2(I_0)$ for which
\begin{align}\label{5.3}
|u(x)|^{2}_{S^{k}_j(I_0)}=\sum_{i=j}^{k}\int_{-1}^{1}\left|{u}^{(i)}(\xi)\right|^2(1-\xi^2)^id\xi<\infty.
\end{align}
For $j=0$, (\ref{5.3}) defines a norm and the corresponding normed space is denoted by $S^{k}(I_0)$.

The following result (Lemma 3.58,~\cite{SCHW}) gives an approximation of $\hat{u}(\xi)$ in the $L^2(I_0)$ norm and $H^1(I_0)$ seminorm.
\begin{lem}\label{lem5.1}
Let $\hat{u}\in H^1(I_0)$ and let
\begin{align}\label{5.4}
b_n=\frac{2n+1}{2}\int_{-1}^{1}\hat{u}^\prime(\xi)L_n(\xi) d\xi
\end{align}
be the Lengedre coefficients of $\hat{u}^\prime$. Then there exists a polynomial $\chi\in S^W(I_0)$ satisfying $\chi(\pm 1)=\hat{u}(\pm 1)$
such that
\begin{subequations}\label{5.5}
\begin{align}
\|\hat{u}^\prime(\xi)-\chi^{\prime}(\xi)\|^{2}_{0,I_0} &=\sum_{n=W}^{\infty} \frac{2|b_n|^2}{(2n+1)},\label{5.5a}\\
\|\hat{u}(\xi)-\chi(\xi)\|^{2}_{0,I_0} &\leq \sum_{n=W}^{\infty} \frac{2|b_n|^2}{n(n+1)(2n+1)},\label{5.5b}\\
\|\hat{u}^\prime(\xi)-\chi^{\prime}(\xi)\|_{0,I_0} &\leq \|u^\prime(x)-\xi^{\prime}(x)\|_{0,I_0}\label{5.5c}
\end{align}
\end{subequations}
for any $\xi\in S^{W}(I_0)$ satisfying $\xi(\pm 1)=\hat{u}(\pm 1)$.
\end{lem}
Similarly, the following Lemma gives an approximation of $\hat{u}(\xi)$ in the $L^2(I_0)$ norm and $H^2(I_0)$ seminorm.
\begin{lem}\label{lem5.2}
Let $\hat{u}\in H^2(I_0)$ and let
\begin{align}\label{5.6}
c_n=\frac{2n+1}{2}\int_{-1}^{1}\hat{u}^{\prime\prime}(\xi)L_n(\xi) d\xi
\end{align}
be the Legendre coefficients of $\hat{u}^{\prime\prime}$. Then there exists a polynomial $\eta\in S^W(I_0)$ satisfying $\eta(\pm 1)=\hat{u}(\pm 1)$
such that
\begin{subequations}\label{5.7}
\begin{align}
\|\hat{u}^{\prime\prime}(\xi)-\eta^{\prime\prime}(\xi)\|^{2}_{0,I_0} &=\sum_{n=W}^{\infty} \frac{2|c_n|^2}{(2n+1)},\label{5.7a}\\
\|\hat{u}(\xi)-\eta(\xi)\|^{2}_{0,I_0} &\leq \sum_{n=W}^{\infty} \frac{2|c_n|^2}{n(n+1)(2n+1)},\label{5.7b}\\
\|\hat{u}^{\prime\prime}(\xi)-\eta^{\prime\prime}(\xi)\|_{0,I_0} &\leq \|u^{\prime\prime}(x)-\xi^{\prime\prime}(x)\|_{0,I_0}\label{5.7c}
\end{align}
\end{subequations}
for any $\xi\in S^{W}(I_0)$ satisfying $\xi(\pm 1)=\hat{u}(\pm 1)$.
\end{lem}
Let $u_{\epsilon}(x)=\exp\left({-\frac{(x+1)}{\epsilon}}\right)$ with $\epsilon>0$. Now, we have the following basic
approximation result on a single element.
\begin{lem}\label{lem5.3}
For any $W\geq 1$, there is a polynomial $\phi_W(x)\in \mathcal{S}^{W}$ such that the estimates
\begin{subequations}\label{5.8}
\begin{align}
|u_{\epsilon}(x)-\phi_W(x)|^{2}_{0,\Omega} &\le \frac{C\epsilon}{W} \left(\frac{e}{(2W+1)\epsilon}\right)^{2W+1},\label{5.8a}\\
|u_{\epsilon}(x)-\phi_W(x)|^{2}_{1,\Omega} &\le \frac{C}{\epsilon} \left(\frac{e}{(2W+1)\epsilon}\right)^{2W+1},\label{5.8b}\\
|u_{\epsilon}(x)-\phi_W(x)|^{2}_{2,\Omega} &\le \frac{C}{\epsilon^3W} \left(\frac{e}{(2W+1)\epsilon}\right)^{2W+1}\label{5.8c}
\end{align}
\end{subequations}
hold, provided $W=O(\frac{1}{\epsilon})$. Here, $C$ denotes a generic constant (independent of $\epsilon$ and $W$).
\end{lem}
\begin{proof}
The proof of the estimates (\ref{5.8a})-(\ref{5.8b}) is provided in Theorem 4.1 of~\cite{SSX1} so we are left to prove
the estimate (\ref{5.8c}). Using approximation results from~\cite{BG1}, there exists a polynomial $\phi_{W}(x)$ of
degree $W$ such that $\phi_{W}(\pm 1)=u_{\epsilon}(\pm 1)$ and
\begin{align}\label{5.9}
|u^{\prime\prime}_\epsilon(x)-\phi^{\prime\prime}_{W}(x)|^{2}_{0,I} &\le \frac{1}{(2W+1)!} |u^{\prime\prime}_{\epsilon}(x)|^{2}_{S^{W}(\Omega)},
\end{align}
where
\[|u^{\prime\prime}_{\epsilon}(x)|^{2}_{S^{W}(\Omega)}=\int_{-1}^{1}\left|{u}^{(W+2)}_\epsilon(\xi)\right|^2(1-\xi^2)^Wd\xi.\]
Now,
\begin{align*}
\left|{u}^{(W+2)}_\epsilon(x)\right|^2&=\frac{1}{\epsilon^{2(W+2)}}\left|\exp\left({-\frac{(x+1)}{\epsilon}}\right)\right|^2
=\frac{1}{\epsilon^{2(W+2)}}e^{\left({\frac{-2(x+1)}{\epsilon}}\right)}.
\end{align*}
Therefore,
\begin{align}\label{5.10}
|u^{\prime\prime}_{\epsilon}(x)|^{2}_{S^{W}(\Omega)}&=\frac{1}{\epsilon^{2(W+2)}}\int_{-1}^{1}e^{\left({\frac{-2(\xi+1)}{\epsilon}}\right)}(1-\xi^2)^Wd\xi
\leq \frac{1}{\epsilon^{2(W+2)}}\int_{-1}^{1}(1-\xi^2)^Wd\xi\notag\\
&\leq \frac{1}{\epsilon^{2(W+2)}}\frac{1}{\sqrt{W+1}}.
\end{align}
Using Stirling's formula
\[n!\approx \sqrt{2\pi n}\:e^{-n}n^n,\]
we get
\begin{align}\label{5.11}
\frac{1}{(2W+1)!}\leq C \left(\frac{e}{2W+1}\right)^{2W+1}\frac{1}{\sqrt{2W+1}},
\end{align}
where $C$ is independent of $W$.

Using, (\ref{5.10})$-$(\ref{5.11}) in (\ref{5.9}), we get
\begin{align*}
|u^{\prime\prime}_\epsilon(x)-\phi^{\prime\prime}_{W}(x)|^{2}_{0,I} &\leq \frac{C}{\epsilon^3} \left(\frac{e}{(2W+1)\epsilon}\right)^{2W+1}\frac{1}{\sqrt{2W+1}}\frac{1}{\sqrt{W+1}}\notag\\
&\leq \frac{C}{\epsilon^3W} \left(\frac{e}{(2W+1)\epsilon}\right)^{2W+1}.
\end{align*}
\end{proof}
Let us denote ${W}_0=W+\frac{1}{2}$ for any integer $W$. Then, it follows from (5.8a)-(5.8c) that the convergence is super-exponential
as $W\rightarrow \infty$ (i.e. in the asymptotic range) provided
\begin{align}\label{5.12}
\frac{e}{2W_0\epsilon}<1\quad \mbox{or}\quad W_0>\frac{e}{2\epsilon}.
\end{align}
\begin{rem}
Note that the inequality (\ref{5.12}) is satisfied for a very large value of $W$ when the boundary layer parameter $\epsilon$ is small.
For instance, if $\epsilon=0.1$ then the inequality holds if $W\approx 13$, if $\epsilon=0.01$ then the inequality holds for $W\approx 136$,
$W\approx 1359$ if $\epsilon=0.001$ and $W\approx 13591$ if $\epsilon=0.0001$. It follows that except for the first two values of $W$
no other value falls in the practical range as the parameter $\epsilon$ decreases.
\end{rem}
We are now in a position to prove our main error estimate theorems for the $p$ and $hp-$versions of the method. First, we
consider error estimates for the $p-$version.
\begin{thm}\label{thm5.1}
Let $u(x)=U(x)$ be the exact solution of the problem (2.2a)-(2.2b) on $\Omega$ and let $z(x)\in \mathcal{S}^{W}(\Omega)$ minimizes
$\mathcal{R}_{p}^{W}(u(x))$ over all $u(x)$. Then there exists a constant $C$ independent of $\epsilon$ and $W$ such that the error
satisfies
\begin{align}\label{5.13}
\|z(x)-U(x)\|^{2}_{2,\epsilon,\Omega} &\leq C\frac{\sqrt{\log W}}{W},
\end{align}
provided $W+\frac{1}{2}>\frac{e}{2\epsilon}$.
\end{thm}
\begin{proof}
Using approximation results from~\cite{BG2}, error estimates from~\cite{KJ} and Lemma \ref{lem5.3}, there exists a polynomial $\phi(x)$
of degree $W$ with $\phi(\pm 1)=u(\pm 1)$ such that
\begin{align}\label{5.14}
\|U(x)-\phi(x)\|^{2}_{2,\epsilon,\Omega} &\le \frac{C_s\epsilon}{W}\left(\frac{e}{(2W+1)\epsilon}\right)^{2W+1}\|U(x)\|^{2}_{s,\epsilon,\Omega},
\end{align}
where $C_s$ is a generic constant independent of $\epsilon$, $W$, and $C_s=ce^{2s}$, $s\leq W$.

We need to estimate the functional $\mathcal{R}_p^{W}(\{\phi\})$ defined by
\begin{align}\label{5.15}
\mathcal{R}_p^{W}(\{\phi\}) &=\|\mathcal{L}\phi(x)-f\|^{2}_{0,\Omega}+|\phi(a)-\alpha|^{2}+|\phi(b)-\beta|^{2}.
\end{align}
Consider the first term $\|\mathcal{L}\phi(x)-f\|^{2}_{0,\Omega}$ on the RHS in (\ref{5.15}) consisting of the sum of residuals in
the differential equation, then using (\ref{5.14}), we get
\begin{align}\label{5.16}
\|\mathcal{L}\phi(x)-f\|^{2}_{0,\Omega} &=\|\mathcal{L}\phi(x)-\mathcal{L}U\|^{2}_{0,\Omega}\leq C\|U(x)-\phi(x)\|^{2}_{2,\epsilon,\Omega}\notag\\
&\leq \frac{C_s\epsilon}{W}\left(\frac{e}{(2W+1)\epsilon}\right)^{2W+1}\|U(x)\|^{2}_{s,\epsilon,\Omega}.
\end{align}
Next, we need to approximate the boundary residuals terms i.e. the terms $|\phi(a)-\alpha|^{2}$ and $|\phi(b)-\beta|^{2}$.
By Sobolev embedding results in one dimension~\cite{ADAM} there exists a constant $C>0$ such that
\[|\phi(x)|\leq C\|\phi(x)\|_{2,\epsilon,\Omega}.\]
Therefore, using (\ref{5.14})
\begin{align}\label{5.17}
|\phi(a)-\alpha|^{2} &=|\phi(a)-U(a)+U(a)-\alpha|^{2} \notag\\
&\leq C\left(\|\phi(a)-U(a)\|^{2}_{2,\epsilon,\Omega}+\|U(a)-\alpha\|^{2}_{2,\epsilon,\Omega}\right) \notag \\
&\leq \frac{C_s\epsilon}{W}\left(\frac{e}{(2W+1)\epsilon}\right)^{2W+1}\|U(x)\|^{2}_{s,\epsilon,\Omega}.
\end{align}
Similarly, we can show that
\begin{align}\label{5.18}
|\phi(b)-\beta|^{2} &\leq \frac{C_s\epsilon}{W}\left(\frac{e}{(2W+1)\epsilon}\right)^{2W+1}\|U(x)\|^{2}_{s,\epsilon,\Omega}.
\end{align}
Combining the inequalities (\ref{5.16}), (\ref{5.17}) and (\ref{5.18}), we get
\begin{align*}
\mathcal{R}_p^{W}(\{\phi\}) &\leq \frac{C_s\epsilon}{W}\left(\frac{e}{(2W+1)\epsilon}\right)^{2W+1}\|U(x)\|^{2}_{s,\epsilon,\Omega}.
\end{align*}
Since $z(x)\in \mathcal{S}^{W}(\Omega)$ minimizes $\mathcal{R}_p^{W}(u(x))$ over all $u(x)$, the estimate
\begin{align*}
\mathcal{R}_p^{W}(\{z\}) &\leq \frac{C_s\epsilon}{W}\left(\frac{e}{(2W+1)\epsilon}\right)^{2W+1}\|U(x)\|^{2}_{s,\epsilon,\Omega}
\end{align*}
holds. Now, using the above estimate in stability estimate in Theorem \ref{thm3.1} by setting $u(x)=z(x)-\phi(x)$, we can write
\begin{align}\label{5.19}
\|z(x)-\phi(x)\|^{2}_{2,\epsilon,\Omega} &\leq \frac{C_s\epsilon}{W}\left(\frac{e}{(2W+1)\epsilon}\right)^{2W+1}\|U(x)\|^{2}_{s,\epsilon,}.
\end{align}
Combining the estimates (\ref{5.14}) and (\ref{5.19}) using triangle inequality, we obtain
\begin{align}\label{5.20}
\|U(x)-z(x)\|^{2}_{2,\epsilon,\Omega} &\leq \frac{C_s\epsilon}{W}\left(\frac{e}{(2W+1)\epsilon}\right)^{2W+1}\|U(x)\|^{2}_{s,\epsilon,\Omega}.
\end{align}
From the differentiability estimate (\ref{3.2}), we can write
\begin{align}\label{5.21}
\|u^{(m)}_\epsilon\|_{0,\Omega}\leq Cd^m\max{\left(m,\frac{1}{\epsilon}\right)^m},\quad \forall \:m\geq 1.
\end{align}
Now,
\begin{align*}
\max{\left(m,\frac{1}{\epsilon}\right)^m}\leq \max{\left(m^m,m!\frac{1}{m!\epsilon^m}\right)}\leq \max{\left(m^m,m!e^{1/\epsilon}\right)}\leq e^mm!\left(Ce^{1/\epsilon}\right),
\end{align*}
where we have used Stirling's formula to obtain the last inequality. Therefore, the differentiability estimate (\ref{5.21}) takes the form
\begin{align}\label{5.22}
\|u^{(m)}_\epsilon\|_{0,\Omega}\leq Cd^m e^me^{1/\epsilon}m!,\quad \forall \:m\geq 1.
\end{align}
This gives,
\begin{align}\label{5.23}
\|U(x)\|^{2}_{s,\epsilon,\Omega}\leq Cd^{2s} \left(e^se^{1/\epsilon}s!\right)^2.
\end{align}
Using the above estimate in equation (\ref{5.20}), we obtain
\begin{align}\label{5.24}
\|U(x)-z(x)\|^{2}_{2,\epsilon,\Omega} &\leq \frac{C_sd^{2s}\epsilon}{W}\left(\frac{e}{(2W+1)\epsilon}\right)^{2W+1}\left(e^se^{1/\epsilon}s!\right)^2.
\end{align}
Let $W_0=W+\frac{1}{2}$ and choose $s=W_0\gamma$, $0<\gamma<1$. Then using Stirling's formula in (\ref{5.24}), we get
\begin{align*}
\|U(x)-z(x)\|^{2}_{2,\epsilon,\Omega} &\leq \frac{Cd^{2\gamma W_0}\epsilon}{W}\left(\frac{e}{2W_0\epsilon}\right)^{2W_0}
\left(e^{2\gamma W_0}e^{2/\epsilon}2\pi\gamma W_0e^{-2\gamma W_0}(\gamma W_0)^{2\gamma W_0}\right)\\
&\leq \frac{C\epsilon}{W}\left(\frac{e}{2W_0\epsilon}\right)^{2W_0}
\left(e^{2/\epsilon}2\pi\gamma W_0(d\gamma W_0)^{2\gamma W_0}\right).
\end{align*}
Since $W_0>\frac{e}{2\epsilon}$, therefore $\left(\frac{e}{2W_0\epsilon}\right)<1$. Choosing $1<W_0<\sqrt{\frac{\log W}{\epsilon}}$
and $\gamma$ such that $d\gamma<1$ there exists a constant $C>0$ such that
\begin{align*}
\|U(x)-z(x)\|^{2}_{2,\epsilon,\Omega} &\leq C\frac{\sqrt{\log W}}{W}.
\end{align*}
\end{proof}
\begin{rem}
It follows that the $p$-version of the method provides super-exponential convergence if $W_0>\frac{e}{2\epsilon}$.
However, this bound is not possible to obtain unless $W$ is very large for small values of the perturbation parameter
$\epsilon$ (For example, $W\approx 13591$ when $\epsilon=10^{-4}$). Thus, the $p$-version over a fixed mesh can yield
only a uniform convergence rate of $O\left(\frac{\sqrt{\log W}}{W}\right)$ which is optimal (upto the factor $\sqrt{\log W}$).

Schwab et al.~\cite{SCHW,SCHW1} have shown that it is also possible to obtain robust exponential convergence using
$p$-version of the finite element method in the pre-asymptotic range $\sqrt{\frac{3}{4\epsilon}}\leq W_0\leq \frac{2}{\epsilon}$
for small values of $\epsilon$. However, we didn't consider the preasymptotic range in our analysis since the proposed
method is essentially a higher order finite element method and the computational results also shows exponential
convergence in our method as well.
\end{rem}
Finally, we prove the error estimate theorem for the $hp-$version of the method.
\begin{thm}\label{thm5.2}
Let $U_{i}(x)=u|_{\Omega_{i}}$ be the exact solution on $\Omega_i$ and $\{z_{i}(x)\}_{i}\in \mathcal{S}^{W}$ minimizes
$\mathcal{R}_{hp}^{W}(\{u_{i}(x)\}_{i})$ over all $\{u_{i}(x)\}_{i}$. Then there exist constants $C, b>0$ independent
of $\epsilon$, $W$ such that the error satisfies
\begin{align}\label{5.25}
\sum_{i=1}^{N}||z_{i}(x)-U_{i}(x)||^{2}_{2,\epsilon,\Omega_{i}} &\le C e^{-bW/\log W},
\end{align}
provided $W=O\left(\frac{1}{\epsilon}\right)$.
\end{thm}
\begin{proof}
Using approximation results from~\cite{BG2}, error estimates from~\cite{KJ} and Lemma \ref{lem5.3}, there exists a
polynomial $\phi_{i}(x)$ of degree $W$ with $\phi_{i}(\pm 1)=u_{i}(\pm 1)$ satisfying
\begin{align}\label{5.26}
\|U_{i}(x)-\phi_{i}(x)\|^{2}_{2,\epsilon,\Omega_{i}} &\le \frac{C_s\epsilon}{W}\left(\frac{e}{(2W+1)\epsilon}\right)^{2W+1}\|U_{i}(x)\|^{2}_{s,\epsilon,\Omega_{i}}
\end{align}
where $C_s$ is a generic constant independent of $\epsilon$, $W$, and $C_s=ce^{2s}$, $s\leq W$.

Consider the set of functions $\{\phi_{i}(x)\}_{i}$. We will estimate the functional
\begin{align}\label{5.27}
\mathcal{R}_{hp}^{W}(\{\phi\}_{i}) &=\sum_{i=1}^{N}\|\mathcal{L}\phi_{i}(x)-f_{i}\|^{2}_{0,\Omega_{i}}+\sum_{i=1}^{N-1}
\left(|[\phi(x_{i})]|^{2}+\left|\left[\frac{d\phi}{dx}(x_{i})\right]\right|^{2}\right)+|\phi(a)-\alpha|^{2}+|\phi(b)-\beta|^{2}.
\end{align}
Consider the first term $\sum_{i=1}^{N}\|\mathcal{L}\phi_{i}(x)-f_{i}\|^{2}_{0,\Omega_{i}}$ on RHS in (\ref{5.27}) which
consists of the sum of residuals in the differential equation, then
\begin{align}\label{5.28}
\|\mathcal{L}\phi_{i}(x)-f_{i}\|^{2}_{0,\Omega_{i}} &=\|\mathcal{L}\phi_{i}(x)-\mathcal{L}U_{i}\|^{2}_{0,\Omega_{i}}\le C\|U_{i}(x)-\phi_i(x)\|^{2}_{2,\epsilon,\Omega_{i}}\notag\\
&\leq \frac{C_s\epsilon}{W}\left(\frac{e}{(2W+1)\epsilon}\right)^{2W+1}\|U_{i}(x)\|^{2}_{s,\epsilon,\Omega_{i}},
\end{align}
using (\ref{5.26}).

Next, we estimate jump terms namely, $|[\phi(x_{i})]|^{2}$ and $\left|\left[\frac{d\phi}{dx}(x_{i})\right]\right|^{2}$.
By continuous Sobolev embedding results for bounded domains in one dimension~\cite{ADAM} (Chapter 4) there exist constants
$C_1>0, C_2>0$ such that
\[|v(x)|\leq C_{1}\|v\|_{2,\epsilon,\Omega},\quad \left|\frac{dv}{dx}(x)\right|\le C_{2}\|v\|_{2,\epsilon,\Omega}\]
and also
\[U_{i}(x_{i})=U_{i+1}(x_{i}),\quad \frac{dU_{i}}{dx}(x_{i})=\frac{dU_{i+1}}{dx}(x_{i}).\]
Therefore, we can write
\begin{align}\label{5.29}
|[\phi(x_{i})]|^{2} &=|\phi_{i}(x_{i})-\phi_{i+1}(x_{i})|^{2} \notag\\
&=|\phi_{i}(x_{i})-U_{i}(x_{i})+U_{i+1}(x_{i})-\phi_{i+1}(x_{i})|^{2} \notag\\
%&\leq C\left(|\phi_{i}(x_{i})-U_{i}(x_{i})|^{2}+|U_{i+1}(x_{i})-\phi_{i+1}(x_{i})|^{2}\right) \notag\\
&\leq C\left(\|\phi_{i}(x)-U_{i}(x)\|^{2}_{2,\epsilon,\Omega_{i}}+\|U_{i+1}(x)-\phi_{i+1}(x)\|^{2}_{2,\epsilon,\Omega_{i+1}}\right) \notag \\
&\leq \frac{C_s\epsilon}{W}\left(\frac{e}{(2W+1)\epsilon}\right)^{2W+1}\left(\|U_{i}(x)\|^{2}_{s,\epsilon,\Omega_{i}}+\|U_{i+1}(x)\|^{2}_{s,\epsilon,\Omega_{i+1}}\right).
\end{align}
Similarly, we obtain
\begin{align}\label{5.30}
\left|\left[\frac{d\phi}{dx}(x_{i})\right]\right|^{2} &\leq \frac{C_s\epsilon}{W}\left(\frac{e}{(2W+1)\epsilon}\right)^{2W+1}
(\|U_{i}(x)\|^{2}_{s,\epsilon,\Omega_{i}}+\|U_{i+1}(x)\|^{2}_{s,\epsilon,\Omega_{i+1}}).
\end{align}
Finally, we need to approximate terms containing boundary residuals. As in the proof of Theorem \ref{thm5.1}, it is easy to see that
\begin{align}\label{5.31}
|\phi(a)-\alpha|^{2} &\le \frac{C_s\epsilon}{W}\left(\frac{e}{(2W+1)\epsilon}\right)^{2W+1}\|U_{1}(x)\|^{2}_{s,\epsilon,\Omega_{1}}, \notag \\
\mbox{and} \quad |\phi(b)-\beta|^{2} &\leq \frac{C_s\epsilon}{W}\left(\frac{e}{(2W+1)\epsilon}\right)^{2W+1}\|U_{N}(x)\|^{2}_{s,\epsilon,\Omega_{N}}.
\end{align}
Combining the inequalities (\ref{5.28})$-$(\ref{5.31}) and using the fact that $N\propto W$ for the $hp$-version, we obtain
\begin{align*}
\mathcal{R}^{W}(\{\phi_{i}\}_{i}) &\leq {C_s\epsilon}\left(\frac{e}{(2W+1)\epsilon}\right)^{2W+1}\left(\sum_{i=1}^{N}\|U_{i}(x)\|^{2}_{s,\epsilon,\Omega_{i}}\right).
\end{align*}
Since $\{z_{i}(x)\}_{i}\in \mathcal{S}^{W}$ minimizes $\mathcal{R}^{W}(\{u_{i}(x)\}_{i})$ over all $\{u_{i}(x)\}_{i}$, therefore the estimate
\begin{align*}
\mathcal{R}_{hp}^{W}(\{z_{i}\}) &\leq {C_s\epsilon}\left(\frac{e}{(2W+1)\epsilon}\right)^{2W+1}\left(\sum_{i=1}^{N}\|U_{i}(x)\|^{2}_{s,\epsilon,\Omega_{i}}\right)
\end{align*}
holds. Now, setting $u_i(x)=z_i(x)-\phi_i(x)$ in the stability estimate (\ref{3.12}) and using above estimate it follows that
\begin{align}\label{5.32}
\sum_{i=1}^{N}\|z_{i}(x)-\phi_{i}(x)\|^{2}_{2,\epsilon,\Omega_{i}} &\leq {C_s\epsilon}\left(\frac{e}{(2W+1)\epsilon}\right)^{2W+1}\left(\sum_{i=1}^{N}\|U_{i}(x)\|^{2}_{s,\epsilon,\Omega_{i}}\right).
\end{align}
Moreover, from (\ref{5.26}) we can write
\begin{align}\label{5.33}
\sum_{i=1}^{N}\|U_{i}(x)-\phi_{i}(x)\|^{2}_{2,\epsilon,\Omega_{i}} &\leq {C_s\epsilon}\left(\frac{e}{(2W+1)\epsilon}\right)^{2W+1}\left(\sum_{i=1}^{N}\|U_{i}(x)\|^{2}_{s,\epsilon,\Omega_{i}}\right).
\end{align}
Combining (\ref{5.32}) and (\ref{5.33}),
\begin{align}\label{5.34}
\sum_{i=1}^{N}\|U_{i}(x)-z_{i}(x)\|^{2}_{2,\epsilon,\Omega_{i}} &\leq {C_s\epsilon}\left(\frac{e}{(2W+1)\epsilon}\right)^{2W+1}\left(\sum_{i=1}^{N}\|U_{i}(x)\|^{2}_{s,\epsilon,\Omega_{i}}\right).
\end{align}
From the differentiability estimate (\ref{5.22}), we have
\begin{align*}%\label{5.35}
\|U_{i}(x)\|^{2}_{s,\epsilon,\Omega_{i}}\leq Cd^{2s} \left(e^se^{1/\epsilon}s!\right)^2.
\end{align*}
Using this in (\ref{5.34}), we get
\begin{align}\label{5.35}
\sum_{i=1}^{N}\|U_{i}(x)-z_{i}(x)\|^{2}_{2,\epsilon,\Omega_{i}} &\leq {C_sd^{2s}\epsilon}\left(\frac{e}{(2W+1)\epsilon}\right)^{2W+1}\left(e^se^{1/\epsilon}s!\right)^2.
\end{align}
Setting $W_0=W+\frac{1}{2}$ such that $W_0>\frac{e}{2\epsilon}$, choosing $s=W_0\gamma$, with $0<\gamma<1$ and using Stirling's
formula $n!\sim \sqrt{2\pi n}e^{-n}n^n$ in (\ref{5.35}), we get %Here, we have used the fact that $N=O(W)$ for the $hp$-version.
\begin{align*}
\sum_{i=1}^{N}\|U_{i}(x)-z_{i}(x)\|^{2}_{2,\epsilon,\Omega_{i}} &\leq Cd^{2\gamma W_0}\epsilon\left(\frac{e}{2W_0\epsilon}\right)^{2W_0}
\left(e^{2\gamma W_0}e^{2/\epsilon}2\pi\gamma W_0e^{-2\gamma W_0}(\gamma W_0)^{2\gamma W_0}\right)\\
&\leq C\epsilon\left(\frac{e}{2W_0\epsilon}\right)^{2W_0}\left(e^{2/\epsilon}2\pi\gamma W_0(d\gamma W_0)^{2\gamma W_0}\right).
\end{align*}
Since $W_0>\frac{e}{2\epsilon}$, therefore $\left(\frac{e}{2W_0\epsilon}\right)<1$. Choosing $1<W_0<\frac{1}{\epsilon\log W}$
and $\gamma$ such that $d\gamma<1$ there exist constants $C, b>0$ such that
\begin{align*}
\sum_{i=1}^{N}\|U_{i}(x)-z_{i}(x)\|^{2}_{2,\epsilon,\Omega_{i}} &\leq Ce^{-\frac{bW}{\log W}}.
\end{align*}
\end{proof}
We now analyze computational complexity of the proposed method for the $p$ and $hp$ versions examined in this article. It is
enough to consider computational complexity for the $hp-$version since the complexity of the $p-$version follows by taking $N=1$.
The computational complexity takes into account number of Degrees of freedom (DoF), cost of assembling the matrix $A^TA$,
preconditioning the matrix using the PCGM and solving the problem i.e., the time taken to solve $A^TAu=A^Tb$. At each iteration
of the PCGM we need to compute action of a matrix on a vector. Thus, the efficiency of computations depends on the ability to
compute $A^TAu$ efficiently and inexpensively for any vector $u$. It can be shown as in~\cite{TOM} that $A^TAu$ can be computed
cheaply without storing the mass and stiffness matrices.

Moreover, we must be able to construct an effective preconditioner for the matrix $A^TA$ so that the condition number of the
preconditioned system is small. This will allow us to compute $A^TAv$ efficiently for any vector $v$ using the PCGM. In Section
4, it has been shown the condition number of the preconditioned system is $O\left(\frac{1}{\epsilon^2}\right)$ for the $p$ and
$hp-$version of the method. The total number of DoF are $O(NW)$ and it costs $O(NW^3)$ operations to compute the matrix vector
product $A^TAu$ (because of the presence of second order derivatives in the least-squares formulation, we require polynomials
of degree less than or equal to $4W$ to evaluate integrals exactly via Gauss-Lobatto-Legendre (GLL) points). Now, it takes
approximately $O\left(\frac{1}{\epsilon}\log\left(\frac{1}{\mu}\right)\right)$ iterations of the PCGM to solve the problem
to a given tolerance $\mu$. Therefore, the overall complexity for the $hp-$version of the method is
$O\left(\frac{NW^3}{\epsilon}\log\left(\frac{1}{\mu}\right)\right)$ which is approximately $O(W^4)$ since $N$ is proportional
to $W$ for the $hp-$version.

\section{Numerical Results}\label{sec6}
In this section, we present numerical results for various one dimensional elliptic boundary layer problems. Both $p$, $hp$
refinements have been investigated for different values of the layer parameter $\epsilon$. Relative error is obtained in the
modified $H^{2}$ norm which is defined as
\begin{align}
\|E\|_{\mbox{rel}} =\frac{\|u_{\mbox{sem}}-u_{\mbox{exact}}\|_{2,\epsilon,\Omega}}{\|u_{\mbox{exact}}\|_{2,\epsilon,\Omega}}\times 100
\end{align}
where, $u_{\mbox{sem}}$ is the spectral element solution obtained from the solver and $u_{\mbox{exact}}$ is the exact solution
of the problem.

Throughout this section, $N$ denotes the number of elements and $W$ denotes the polynomial order used for approximation. For the
$p$-version, we have taken $N=1$ and increased polynomial order $W$ to achieve accuracy in all examples. For the $hp$-version,
we have used a parallel computer and the number of elements taken to be proportional to $W$. Moreover, each element is mapped
onto a single processor. In all plots of this Section the norm of the relative error $\|E\|_{\mbox{rel}}$ is plotted against the
polynomial order on a semi-log scale. Test cases are performed for the values of the boundary layer parameter $\epsilon$ ranging
between $\epsilon=0.1$ and $\epsilon=10^{-4}$ that are specified at appropriate places. The numerical results have been obtained
with FORTRAN-90 sequential and parallel codes and in case of parallel codes the library used for inter processor communication is
Message Passing Interface (MPI).

All computations are carried out on a distributed memory HPC cluster which consists of 6-nodes (1 master and 5 slaves) and each
node is a dual core Intel 276 Xeon 4116 processor with 20 cores on each processor. The memory available on each processor is 96GB
along with 24TB IB enables storage capacity.
\begin{rem}
Note that, in all examples, we have computed percentage relative error using the modified (parameter dependent) $H^2$ Sobolev norm
$\|\cdot\|_{2,\Omega,\epsilon}$ because the proposed method formulate the given problem as a standard least-squares minimization
problem which computes residuals in modified $H^2-$ norms (unlike the Galerkin approach, there is no weak formulation or splitting
of the problem as a system of two first order ODEs to reduce the order) so the best possible approximation error will be in the
modified Sobolev norm $\|\cdot\|_{2,\Omega,\epsilon}$.
\end{rem}
%=======================================================
\subsection{\textbf{Example 1: (Boundary layer at left end)}}
Consider the model elliptic boundary layer problem~\cite{SW} on $\Omega=(0,1)$
\begin{align}\label{ex1}
-\epsilon^{2}u_{xx}+u&=(1-\epsilon^{2})e^x-x\left(e+e^{-1/\epsilon}\right)-2(1-x), \quad x\in\Omega,\\
u(0)&=u(1)=0.
\end{align}
The exact solution for problem is
\begin{align*}
u(x)= e^{-x/\epsilon}+e^x-x\left(e+e^{-1/\epsilon}\right)-2(1-x).
\end{align*}
This problem has a boundary layer at the left end point $x=0$ as $\epsilon\rightarrow0$ (see Figure \ref{soln-ex1}).
\begin{figure}[ht]
\centering
\includegraphics[width=0.35\linewidth]{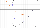}
\caption{Exact solution of model problem in Example 6.1 having boundary layer at $x=0$.}
\label{soln-ex1}
\end{figure}
Plots for the performance of the $p$ and $hp$-version are displayed in Figure \ref{ex1-error}. Figure \ref{ex1a} shows
the relative error in $H^{2}$ norm versus polynomial order for $p$ version for different values of $\epsilon$. For
$\epsilon=0.1$, relative error decays from $O(10^{1})$ to $O(10^{-5})$ by increasing polynomial order from $W=4$ to
$W=20$ over a single element and no improvement takes place by further increase in $W$. For $\epsilon=0.01$, the relative
error decays upto $O\left(10^{-4}\right)$ for $W=40$. On choosing $\epsilon=0.001$ and $\epsilon=0.0001$, the Figure shows
that the error saturates quickly being of order $O(10^{-1})$ as $W$ increases. Therefore, the $p-$version is not able to
resolve the boundary layer at smaller values of the parameter $\epsilon$ as expected.
\begin{figure}[ht]
\centering
\subfigure[Performance of the $p$-version]{
\includegraphics[width=0.45\linewidth]{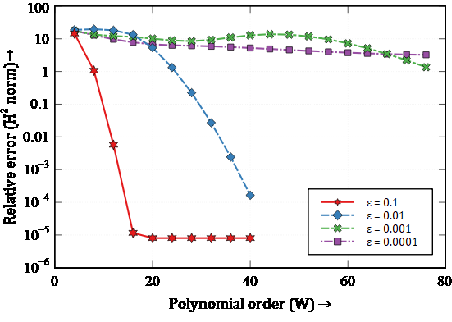}
\label{ex1a}
}
\subfigure[Performance of the $hp$-version]{
\includegraphics[width=0.45\linewidth]{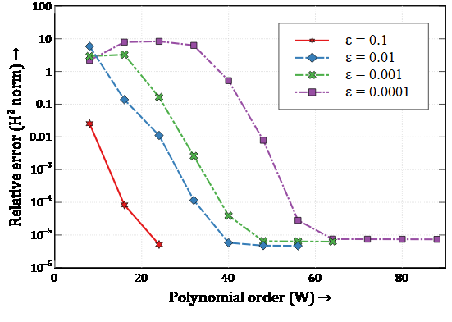}
\label{ex1b}
}
\caption{Comparison of the $p$ and $hp$ version for Example 6.1.}
\label{ex1-error}
\end{figure}

Asymptotically from Lemma \ref{lem5.3}, the $p$-version on a single element will have the best convergence rate for any fixed
$\epsilon$. However, this asymptotic convergence is not visible as seen in Figure \ref{ex1a}. Infact by Lemma \ref{lem5.3},
the error will be in the super-exponential (asymptotic) range for $W>\frac{e}{2\epsilon}$. However, this bound is possible to
achieve only for un-realistically large values of $W$ even for moderate $\epsilon$ (e.g., $W\approx 13591$ for $\epsilon=0.0001$).
Notice that, the graph in Figure \ref{ex1a} for $\epsilon=0.1$ is nearly a straight line for $W$ in this range which confirms
the theoretical estimate in Lemma 5.3.

For the case when $\epsilon=0.01$, we note that the graph for $\log \|E\|_{\mbox{rel}}$ is a parabolic curve for $W$ large enough.
This is due to the fact that this curve is plotted for $\epsilon=0.01$ in the preasymptotic range of $W$. In this case,
$\sqrt{\frac{3}{4\epsilon}}\leq W_0\leq \frac{2}{\epsilon}$, i.e., $5\leq W_0\leq 200$. It is clear from the graph that
the convergence is observed in the pre-asymptotic range as well, though we have not provided error bounds for this range
in our analysis. As $\epsilon$ decreases further, the error is seen to deteriorate.

Let us now consider the $hp$-version (i.e., the $p$-version on a variable mesh). Figure \ref{ex1b} shows the relative error
in $H^{2}$ norm versus polynomial order $W$ for $hp$-version for $\epsilon=0.1$ to $\epsilon=0.0001$. Clearly, the error decays
exponentially and obey the estimate in Theorem \ref{thm5.2} because the graph of $\log(\|E\|_{\mbox{\mbox{rel}}})$ against $W$ is almost
a straight line. For $\epsilon=0.001$, uniform $hp$ refinement brings the relative error down to $O(10^{-6})$, which is
significant good compared to the $p$-version for the same $\epsilon$. Thus, the $hp$-version performs significantly better
than the $p$-version.
%======================================================================
\subsection{\textbf{Example 2: (Boundary layer at right end)}}
Next, we consider a reaction-diffusion boundary layer problem
\begin{align}\label{ex2}
-\epsilon^{2}u_{xx}+u&=-\frac{x+1}{2}, \quad x\in\Omega=(-1,1),\\
u(1)&=u(-1)=0.
\end{align}
The exact solution to this problem is given by
\begin{align*}
u(x)= \frac{\sinh(x+1)/\epsilon}{\sinh(2/\epsilon)}-\frac{x+1}{2}.
\end{align*}
There is a boundary layer of width $O(\epsilon)$ at the outflow boundary $x=1$ as $\epsilon\rightarrow 0$ (see Figure \ref{soln-ex2}).
This example was also considered~\cite{CA,LS,SSX1}.
\begin{figure}[ht]
\centering
\includegraphics[width=0.35\linewidth]{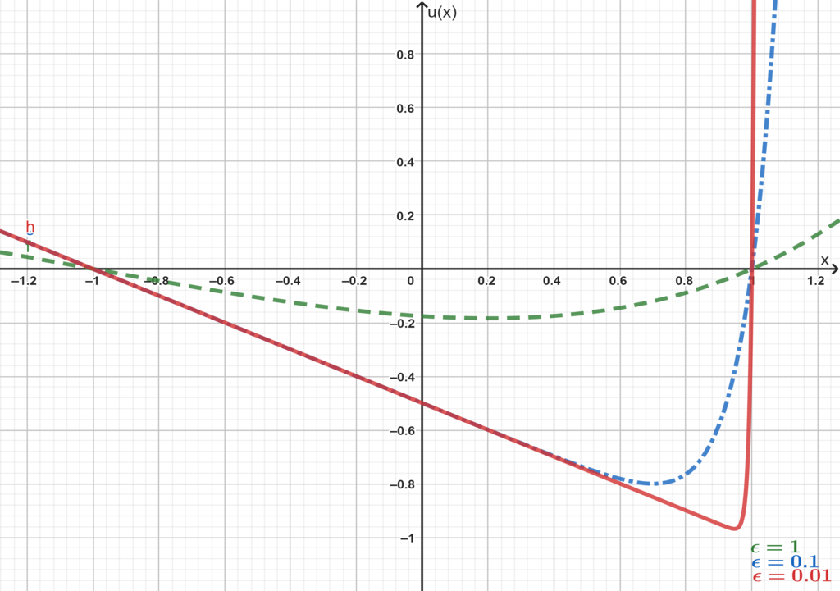}
\caption{Exact solution for Example 6.2 having a boundary layer at the outflow boundary $x=1$.}
\label{soln-ex2}
\end{figure}

Plots for simulation results are depicted in Figure \ref{ex2-error}. Figure \ref{ex2a} shows the relative error plotted against
polynomial order $W$ for the $p$-version. The error decreases to $O(10^{-5})$ with polynomial order $W=20$ and $\epsilon=0.1$.
Relative error reaches $O(10^{-2})$ for $\epsilon=0.01$ on choosing $W$ between $4$ to $40$. With $\epsilon=0.001$, relative
error deteriorates badly and reached only $O(10^{1})$ for $W=40$ and below. A very similar order of convergence was observed
in~\cite{SSX1}.
\begin{figure}[ht]
\centering
\subfigure[Performance of the $p$-version]{
\includegraphics[width=0.45\linewidth]{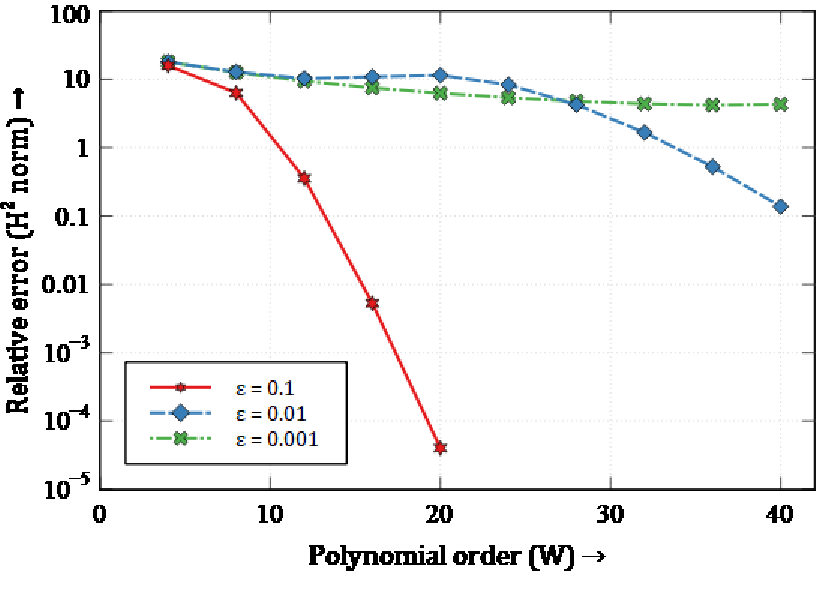}
\label{ex2a}
}
\subfigure[Performance of the $hp$-version]{
\includegraphics[width=0.45\linewidth]{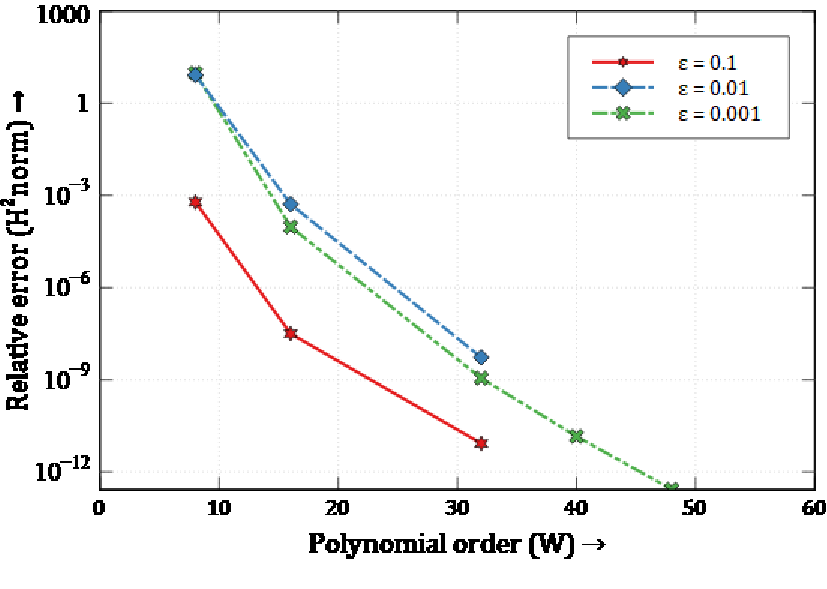}
\label{ex2b}
}
\caption{Comparison of the $p$ and $hp$ version for Example 6.2.}
\label{ex2-error}
\end{figure}

In Figure \ref{ex2b}, the relative error is plotted against $W$ for the $hp$-version. We notice that the error decays exponentially
and for $\epsilon=0.1$, the relative error is of order $O(10^{-12})$ with $W=32$ which is in agreement with our robust error estimate
in Theorem \ref{thm5.2}. For $\epsilon=0.01$, we can see that the error decays to nearly $O(10^{-9})$ at the cost of polynomial order
$W=32$ and the decay in error improves significantly down to $O(10^{-12})$ when $\epsilon=0.001$.
%------------------------------------------------
This remarkable accuracy and improvement in the convergence rate for smaller values of $\epsilon$ is better than the $hp$-FEM proposed
in~\cite{SSX1} and not possible to obtain with the methods based on a single element presented in~\cite{CA,LS}. Thus, our method seems
to outperform previous approaches based on $p$-FEM and $hp$-FEM with a fixed number of elements.

%==============================================================
\subsection{\textbf{Example 3: (Boundary layer at both end points)}}
%==============================================================
Consider the model boundary layer problem
\begin{align}\label{ex3}
-\epsilon^{2}u_{xx}+u&=1, \quad \mbox{in} \quad \Omega=(-1,1),\\
u(-1)&=u(1)=0.
\end{align}
Here, $0<\epsilon\leq 1$ is the boundary layer parameter. The exact solution of the problem is given by
\begin{align*}
u(x)=1-\frac{\cosh(x/\epsilon)}{\cosh(1/\epsilon)}.
\end{align*}
There are two boundary layers near the end points $x=1$ and $x=-1$ (see Figure \ref{soln-ex3} below).
\begin{figure}[ht]
\centering
\includegraphics[width=0.35\linewidth]{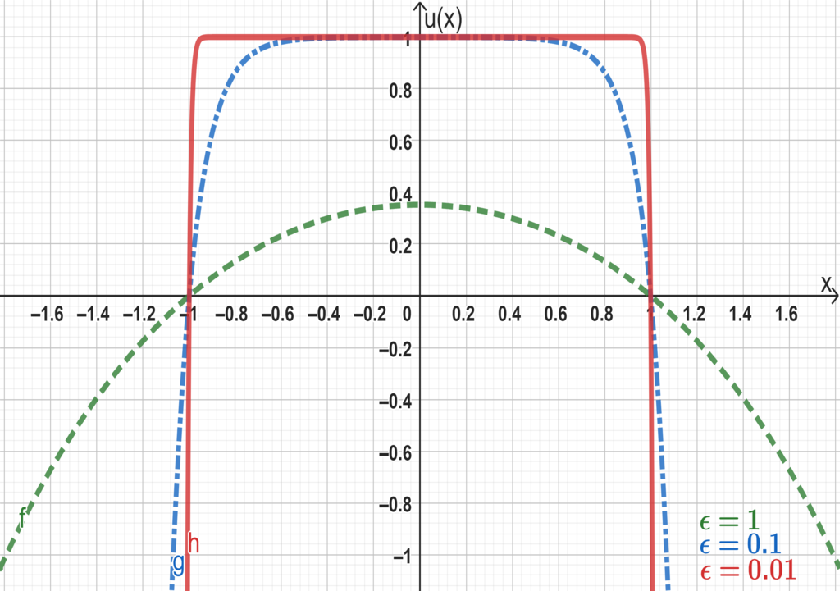}
\caption{Exact solution for Example 6.3 having layers at both ends of the domain.}
\label{soln-ex3}
\end{figure}

This example was extensively examined in~\cite{SCHW1,XENO1} using $h$, $p$, $hp$ and $rp$ methods in the FEM framework
with uniform and non-uniform meshes. Figure \ref{ex3-error} displays relative error with different values of $\epsilon$
using the proposed methods. Figure \ref{ex3a} shows relative error for the $p$-version on a single element and with
variable polynomial order $W$ and Figure \ref{ex3b} shows relative error for the $hp$-version with simultaneous varying
polynomial order and decreasing the mesh size.
%------------------------------------------------------------
\begin{figure}[ht]
\centering
\subfigure[Performance of the $p$-version]{
\includegraphics[width=0.45\linewidth]{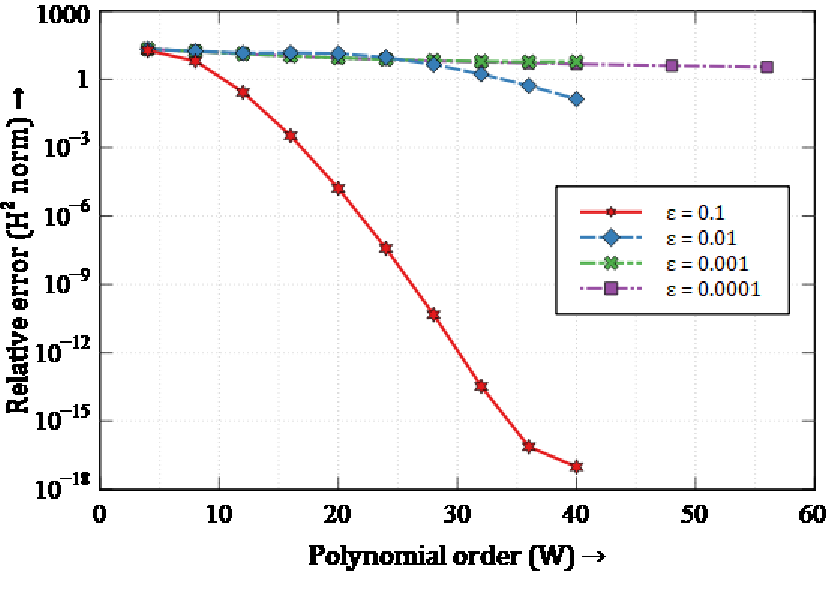}
\label{ex3a}
}
\subfigure[Performance of the $hp$-version]{
\includegraphics[width=0.45\linewidth]{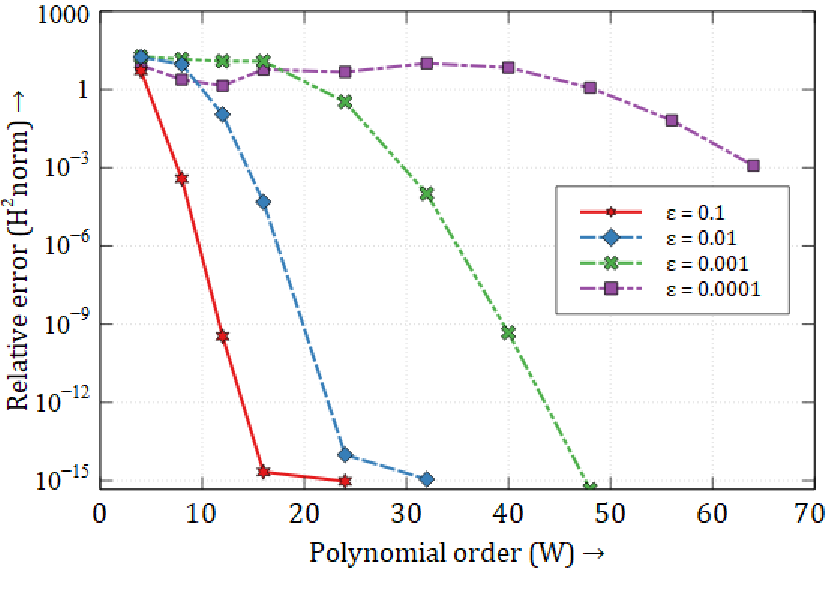}
\label{ex3b}
}
\caption{Comparison of the $p$ and $hp$ version for Example 6.3.}
\label{ex3-error}
\end{figure}
%-----------------------------------------------------
First we discuss the performance of the $p$-version. For $\epsilon=0.1$, increasing polynomial order from $W=4$ to $W=40$,
we see that the relative error decays super-exponentially and reaches approximately $O(10^{-17})$ as predicted by the Lemma
\ref{lem5.3}. In~\cite{SCHW1,XENO1} a much less convergence rate was observed with a comparable number of degrees of freedom
using the $p$-version FEM on a single element. For $\epsilon=0.01$, we observe that the relative error remains almost constant
of order of $O(10^{-2})$ with $W=40$ and deteriorates further when $\epsilon$ is 0.001 and 0.0001 and the error graphs overlap
each other for the later two values of epsilon. Thus, the $p$-version is not able to resolve the boundary layer at all for
small $\epsilon$ values.

Therefore, we tested the problem for the $hp$-version and from Figure \ref{ex3b} it is clear that the $hp$-version delivers
exponential convergence rate with less than half the polynomial order than that used by the $p$-version for $\epsilon=0.1$.
Decreasing $\epsilon$ further, we see that for $\epsilon=0.01, 0.001$ $hp$-version still delivers exponential accuracy with
fewer degrees of freedom as compared with the $hp$-FEM in~\cite{SCHW1,XENO1}.
%=============================================================
\subsection{\textbf{Example 4: (Boundary layer at right end point for convection-diffusion problem)}}
In our last example, we consider a model convection-diffusion problem as follows:
\begin{align}\label{ex4}
-\epsilon^{2} u_{xx}+u_{x}&=f \quad \mbox{in} \quad \Omega=(-1,1),\\
u(-1)&=u(1)=0.
\end{align}
The exact solution of this problem is
\begin{align*}
u(x)= \frac{e^{\frac{x+1}{\epsilon}}-1}{e^{\frac{2}{\epsilon}}-1}-\frac{x+1}{2}.
\end{align*}
The problem has a boundary layer at the right boundary $x=1$ as $\epsilon\rightarrow0$~\cite{SSX1} (see also Figure \ref{soln-ex4}).
%-------------------------------------------------------
\begin{figure}[ht]
\centering
\includegraphics[width=0.35\linewidth]{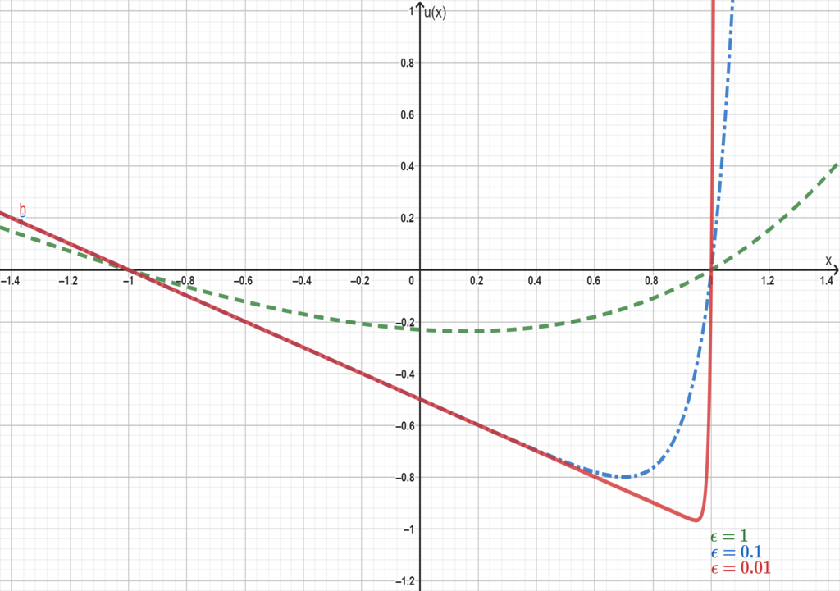}
\caption{Exact solution for Example 6.4 having boundary layer at right end.}
\label{soln-ex4}
\end{figure}
%------------------------------------------------------
Figure \ref{ex4-error} presents graphs for relative errors with different values of $\epsilon$ versus the polynomial orders.
Figure \ref{ex4a} displays relative error for the $p$-version on a single element and \ref{ex4b} shows relative error for the
$hp$-version with a variable mesh.

For the $p$-version with $\epsilon=0.1$, relative error decreases upto $O(10^{-18})$ at the cost of polynomial order $W=40$.
When $\epsilon=0.01$, then the errors decays upto $O(10^{-6})$ for $W=64$. Further lower values of $\epsilon$ doesn't improve
the performance of the $p$-version since the second order derivatives are becoming more singular.

For the $p$-version we observe accuracy of the order of $O(10^{-12})$ for polynomial orders $W=32$ and $W=40$ corresponding
to $\epsilon=0.1$ and $\epsilon=0.01$ respectively. If we decrease layer parameter further to $\epsilon=0.001$ the $hp$
version still recovers exponential accuracy with a polynomial order $W=48$.
%------------------------------------------------------------
\begin{figure}[ht]
\centering
\subfigure[Performance of the $p$-version]{
\includegraphics[width=0.45\linewidth]{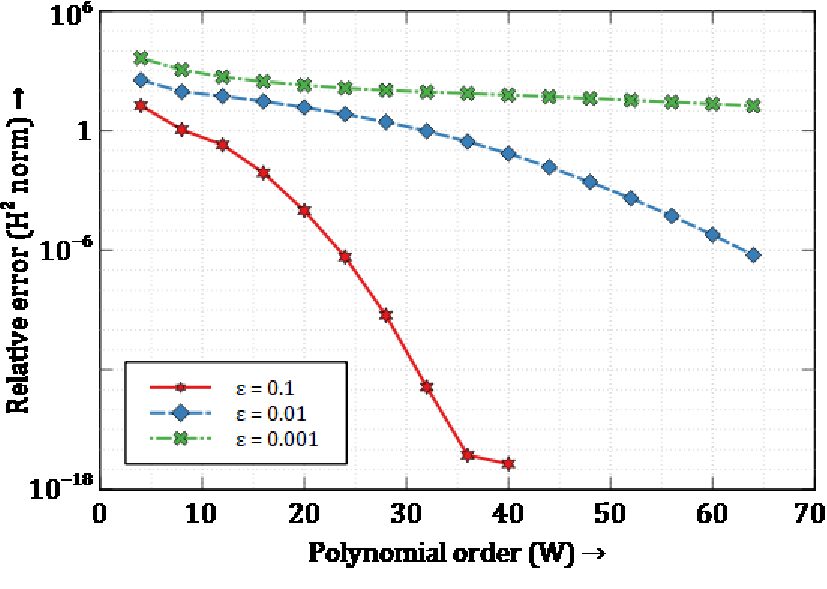}
\label{ex4a}
}
\subfigure[Performance of the $hp$-version]{
\includegraphics[width=0.45\linewidth]{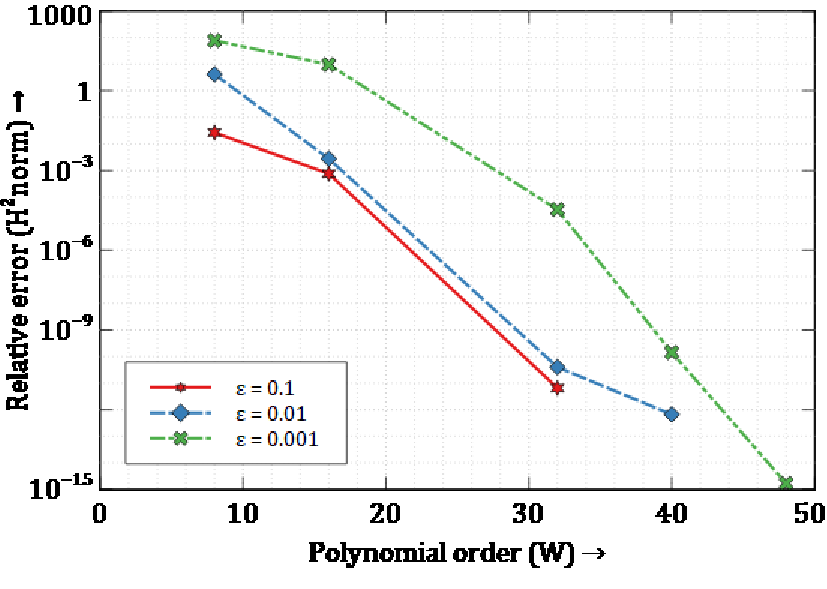}
\label{ex4b}
}
\caption{Comparison of the $p$ and $hp$ version for Example 6.4.}
\label{ex4-error}
\end{figure}
%-----------------------------------------------------

In summary, the errors obtained from the $hp$-version in all examples are found to be consistently the smallest
as compared to those obtained by the $p$-version. This confirms that the $hp$-version is extremely robust and
efficient even for very small values of the boundary layer parameter $\epsilon$ because we recovered relative
errors in the range of order $10^{-15}$ to $10^{-18}$ with polynomial orders just approximately 48 in the worst
case scenario. We remark that although we have presented relative errors in the $H^2$-norm, the relative errors
in the energy norm and pointwise errors will behave similarly. Our numerical using the $p$ and $hp$ methods are
comparable very closely with those obtained in~\cite{CA,LS,SCHW1,SSX1,XENO1} and references therein for same set
of problems.

\section{Conclusions}\label{sec7}
In this article, we presented a parameter robust, exponentially accurate least-square spectral element (LSQ-SEM)
method for one-dimensional elliptic boundary layer problems using $p$ and $hp$ spectral element methods. We derived
stability estimates for non-conforming spectral elements and designed numerical schemes which are able to resolve the
boundary layers uniformly in $\epsilon$ at a robust rate which is of $O\left(\frac{\sqrt{\log W}}{W}\right)$ for the
$p$-version, where $W$ denotes the polynomial order. For the $hp$-version of the method we proved exponential decay
in the error at a rate which is $O\left(e^{-{W}/{\log W}}\right)$. Our estimates and results are supplemented
with extensive numerical computations to validate the theory. Numerical results show good agreement with convergence
results with a few degrees of freedom even when $\epsilon$ is as small as $10^{-4}$. The overall complexity of the
method for the $hp-$version to obtain solution for a given tolerance $\mu$ is $O\left(\frac{NW^3}{\epsilon}\log
\left(\frac{1}{\mu}\right)\right)\approx O(W^4)$ since $N\propto W$ for the $hp-$version.

For the $p$-version we observed super-exponential convergence when $W>\frac{e}{2\epsilon}$. The method is able
to resolve boundary layers for elliptic problems with layers at both ends and reaction-diffusion problems with
layer at one end. The method is superior to existing methods including lower order $h$ and $p$ version of FEM
with optimal `exponential' mesh refinement. A Jacobi type preconditioner can also be used for the PCGM for ease
of implementation.

We examined a simple boundary layer problem, however, the theory presented here is applicable to more general
situations. In future, we plan to extend this theory to other important discretizations in order to design
numerical schemes using spectral element methods which will resolve boundary layers at a fully robust exponential
rate independent of the boundary layer parameter.
It has been proved~\cite{M1,M2,MS3,SCHW1,SSX1,XENO1,XENO2} that this is possible to achieve by adding only two
extra element of size $O(p\epsilon)$ near the boundary layers in the framework of finite element methods using
a boundary layer mesh called the $rp$-mesh. We plan to consider similar mesh refinements techniques combined
with geometric, $p$, and $hp$ refinements to give a complete analysis and resolution of boundary layers using
spectral element methods. Problems with variable coefficients, discontinuous coefficients and problems containing
interior layers will also be addressed.

\section*{Data Availability Statement}
Data sharing is not applicable to this article as no new data were created or analyzed in this study.

\section*{Conflict of Interest}
The authors declare there is no conflict of interest.

\section*{Author Contributions}
\textbf{Akhlaq Husain:} Conceptualization, Methodology, Supervision, Writing- Original draft preparation.
\textbf{Aliya Kazmi:} Visualization, Investigation, Software, Validation, Writing- Original draft preparation.
\textbf{Subhashree Mohapatra:} Investigation, Software, Writing- Reviewing and Editing.
%\textbf{Mohammad Sajid:} Validation, Supervision, Writing- Reviewing and Editing.
\textbf{Ziya Uddin:} Supervision.

\section*{Declaration of AI use}
We have not used AI-assisted technologies in creating this article.

\end{document}